\documentclass[a4paper,12pt]{preprint}
\usepackage[margin=1.3in]{geometry}  
\usepackage[full]{textcomp}
\usepackage[osf]{newtxtext}

\usepackage{mathalfa}
\usepackage{microtype}

\usepackage{enumitem}
\usepackage{amsmath}
\usepackage{amsfonts} 
\usepackage{amssymb}             

\usepackage{comment}
\usepackage{breakurl}

\usepackage{mathrsfs}
\usepackage{booktabs}
\usepackage{mathtools}

\usepackage{tikz}
\usepackage{amsthm}                
\usepackage{mhequ}
\usepackage{hyperref}

\usepackage{accents}

\hypersetup{pdfstartview=}

\addtocontents{toc}{\protect\hypersetup{hidelinks}}

\DeclareSymbolFont{sfoperators}{OT1}{ptm}{m}{n}
\DeclareSymbolFontAlphabet{\mathsf}{sfoperators}

\makeatletter
\def\operator@font{\mathgroup\symsfoperators}
\makeatother
\newcommand{\eqdef}{\stackrel{\mbox{\tiny def}}{=}}

\numberwithin{equation}{section}

\newtheorem{thm}{Theorem}[section]

\newtheorem{lem}[thm]{Lemma}
\newtheorem{prop}[thm]{Proposition}

\newtheorem{cor}[thm]{Corollary}

\newtheorem{assumption}[thm]{Assumption}
\theoremstyle{remark}
\newtheorem{remark}[thm]{Remark}
\newtheorem{rmk}[thm]{Remark}

\makeatletter
\def\th@newremark{\th@remark\thm@headfont{\bfseries}}

\def\bdiamond{\mathop{\mathpalette\bdi@mond\relax}}
\newcommand\bdi@mond[2]{%
	\vcenter{\hbox{\m@th
			\scalebox{\ifx#1\displaystyle 2.6\else1.8\fi}{$#1\diamond$}%
	}}%
}

\def\bDiamond{\mathop{\mathpalette\bDi@mond\relax}}
\newcommand\bDi@mond[2]{%
	\vcenter{\hbox{\m@th
			\scalebox{\ifx#1\displaystyle 2.6\else1.2\fi}{$#1\Diamond$}%
	}}%
}
\makeatletter

\definecolor{darkgreen}{rgb}{0.1,0.7,0.1}
\definecolor{darkred}{rgb}{0.7,0.1,0.1}
\definecolor{darkblue}{rgb}{0,0,0.7}
\newcommand{\EE}{\mathbb{E}}
\newcommand{\NN}{\mathbb{N}}
\newcommand{\PP}{\mathbb{P}}
\newcommand{\RR}{\mathbb{R}}

\newcommand{\bB}{\mathcal{B}}

\newcommand{\dD}{\mathcal{D}}

\newcommand{\eE}{\mathcal{E}}
\newcommand{\fF}{\mathcal{F}}
\newcommand{\gG}{\mathcal{G}}
\newcommand{\hH}{\mathcal{H}}
\newcommand{\iI}{\mathcal{I}}
\newcommand{\kK}{\mathcal{K}}

\newcommand{\nN}{\mathcal{N}}

\newcommand{\qQ}{\mathcal{Q}}

\newcommand{\R}{\mathcal{R}}
\newcommand{\sS}{\mathcal{S}}
\newcommand{\tT}{\mathcal{T}}

\newcommand{\xX}{\mathcal{X}}

\newcommand{\yY}{\mathcal{Y}}

\newcommand{\1}{\mathbf{1}}

\newcommand{\LLL}{\mathcal{L}}

\newcommand{\eps}{\varepsilon}

\newcommand{\FFF}{\mathcal{F}}
\newcommand{\uem}{u_{m,\eps}}
\newcommand{\vem}{v_{m,\eps}}

\makeatletter 
\newcommand{\cov}{{\operator@font cov}}
\newcommand{\var}{{\operator@font var}}
\newcommand{\corr}{{\operator@font corr}}
\newcommand{\diam}{{\operator@font diam}}
\newcommand{\Av}{{\operator@font Av}}
\newcommand{\trig}{{\operator@font trig}}
\newcommand{\Enh}{{\operator@font Enh}}
\makeatother

\colorlet{symbols}{blue!90!black}
\colorlet{testcolor}{green!60!black}

\def\${|\!|\!|}

\makeatletter
\def\DeclareSymbol#1#2#3{\expandafter\gdef\csname MH@symb@#1\endcsname{\tikz[baseline=#2,scale=0.15,draw=symbols]{#3}}\expandafter\gdef\csname MH@symb@#1s\endcsname{\scalebox{0.7}{\tikz[baseline=#2,scale=0.15,draw=symbols]{#3}}}}
\def\<#1>{\csname MH@symb@#1\endcsname}
\makeatother

\setlist[itemize]{topsep=3pt,itemsep=1.5pt,parsep=0pt}

\def\scal#1{\langle#1\rangle}
\def\cent#1{\mathopen{{\langle\kern-0.3em\rangle}}#1\mathclose{{\langle\kern-0.3em\rangle}}}

\def\d{\partial}

\begin{document}

\title{Global well-posedness for the mass-critical stochastic nonlinear Schr\"{o}dinger equation on $\RR$: small initial data}
\author{Chenjie Fan$^1$ and Weijun Xu$^2$}
\institute{University of Chicago, US, \email{cjfanpku@gmail.com}
\and University of Warwick, UK / NYU Shanghai, China, \email{weijunx@gmail.com}}

\maketitle

\begin{abstract}
	We prove the global existence of solution to the small data mass critical stochastic nonlinear Schr\"{o}dinger equation in $d=1$. We further show the stability of the solution under perturbation of initial data. Our construction starts with the existence of the solution to the truncated subcritical problem. We then obtain uniform bounds on these solutions that enable us to reach criticality and then remove the truncation. 
\end{abstract}

\setcounter{tocdepth}{2}
\microtypesetup{protrusion=false}
\tableofcontents
\microtypesetup{protrusion=true}

\section{Introduction}

\subsection{The problem and the main statement}

The aim of this article is to show global existence of the solution to the one dimensional stochastic nonlinear Schr\"{o}dinger equation
\begin{equation} \label{eq:main_eq}
i \partial_t u+\Delta  u = |u|^4 u + u \circ \frac{{\rm d} W}{{\rm d} t}, \qquad x \in \RR,\; t \geq 0
\end{equation}
with small initial data $u_0$. Here, $\frac{{\rm d} W}{{\rm d}t}$ is a real-valued Gaussian process that is white in time and coloured in space. The precise assumption on the noise will be specified below. In the equation, $\circ$ denotes the Stratonovich product, which is the only choice of the product that preserves the $L^2$-norm of the solution. In practice, it is more convenient to treat the equation in the It\^{o} form. In order to be more precise about the equation and the It\^{o}-Stratonovich correction, we give the precise definition of the noise below. 

Let $(\Omega, \FFF, \PP)$ be a probability space, $(\FFF_t)_{t \geq 0}$ be a filtration and $\{B_k\}_{k \in \NN}$ be a sequence of independent standard Brownian motions adapted to $(\FFF_t)$. Fix a set of orthonormal basis $\{e_k\}$ of $L^{2}(\RR)$ and a linear operator $\Phi$ on it. The process $W$ is
\begin{equation*}
W(t,x) := \sum_{k \in \NN} B_{k}(t) \cdot (\Phi e_k)(x). 
\end{equation*}
$W$ is then a Gaussian process on $L^{2}(\RR)$ with covariance operator $\Phi \Phi^{*}$, which can be formally written as $W = \Phi \tilde{W}$ where $\tilde{W}$ is the cylindrical Wiener process. Our assumption on the operator $\Phi$ is the following.

\begin{assumption} \label{as:Phi}
	We assume $\Phi: L^{2}(\RR) \rightarrow \hH$ is a trace-class operator, where $\hH$ is the Hilbert space of real-valued functions with the inner product
	\begin{equation*}
	\scal{f,g}_{\hH} = \sum_{j=0}^{M} \scal{(1+|x|^{K}) f^{(j)}, (1+|x|^{K})g^{(j)}}_{L^{2}}
	\end{equation*}
	for some sufficiently large $K$ and $M$. 
\end{assumption}

Clearly $\hH \hookrightarrow W^{1,p}$ for every $p \in [1,+\infty]$, and our assumption of $\Phi$ implies $\Phi$ is $\gamma$-radonifying from $L^{2}$ to $W^{1,p}$ for every $p \in [1,+\infty]$ ($K, M \geq 10$ would be sufficient for our purpose). A typical example of such an operator is $\Phi e_0 = V(x)$ for some nice function $V$ while $\Phi$ maps all other basis vectors to $0$. In this case, the noise is $B'(t) V(x)$ where $B(t)$ is the standard Brownian motion. 

\begin{rmk}
	With such spacial smoothness of the noise, one may wonder whether the solution theory for \eqref{eq:main_eq} follows directly from the deterministic case. This turns out to be not the case. In fact, the main issue is that the nonlinearity and randomness in the natural solution space together prevent one from setting up a usual fixed point problem. See Sections~\ref{sec:difficulty} and~\ref{sec:strategy} below for more detailed discussions. 
\end{rmk}

Now, we re-write \eqref{eq:main_eq} in its It\^{o} form as
\begin{equation*}
i \partial_t u + \Delta u = |u|^4 u + u \dot{W} - \frac{i}{2} u F_{\Phi}, 
\end{equation*}
where $\dot{W} = \frac{{\rm d} W}{{\rm d} t}$, and
\begin{equation*}
F_{\Phi}(x) = \sum_{k} (\Phi e_k)^{2}(x)
\end{equation*}
is the It\^{o}-Stratonovich correction. Note that $F_{\Phi}$ is independent of the choice of the basis. The assumption on $\Phi$ guarantees that $\|F_{\Phi}\|_{W^{1,p}_{x}} < \infty$ for all $p \in [1,+\infty]$. 

Before we state the main theorem, we introduce a few notations. Let $\sS(t) = e^{i t \Delta}$ be the linear propagator. For every interval $\iI$, let
\begin{equation} \label{eq:space}
\xX_{1}(\iI) = L_{t}^{\infty} L_{x}^{2}(\iI) := L^{\infty}(\iI, L^{2}(\RR)), \quad \xX_{2}(\iI) = L_{t}^{5} L_{x}^{10}(\iI) := L^{5}(\iI, L^{10}(\RR)), 
\end{equation}
and $\xX(\iI) = \xX_{1}(\iI) \cap \xX_{2}(\iI)$ in the sense that $\|\cdot\|_{\xX(\iI)} = \|\cdot\|_{\xX_{1}(\iI)} + \|\cdot\|_{\xX_{2}(\iI)}$. We also write $L_{\omega}^{\rho} \xX$ as an abbreviation for $L^{\rho}(\Omega,\xX)$. Our main statement is then the following. 

\begin{thm} \label{th:main}
There exists $\delta_0 > 0$ and $T_0 > 0$ depending on $\delta_0$ only such that for every $u_0$ independent of the filtration $(\fF_t)$ with $\|u_0\|_{L_{\omega}^{\infty}L_x^2} < \delta_0$ and every $\rho_0>5$, there exists a unique process $u \in L_{\omega}^{\rho_0} \xX(0,T_0)$ adapted to the filtration $\FFF_{\leq T_0}$ such that
	\begin{equation} \label{eq:duhamel_main}
	\begin{split}
	u(t) &= \sS(t) u_0 - i \int_{0}^{t} \sS(t-s) \big( |u(s)|^{4} u(s) \big) {\rm d} s\\
	&- i \int_{0}^{t} \sS(t-s) u(s) {\rm d} W_s - \frac{1}{2} \int_{0}^{t} \sS(t-s) \big( F_{\Phi} u(s) \big) {\rm d} s, 
	\end{split}
	\end{equation}
	where the above integral is in the It\^{o} sense, and there exists $C>0$ such that
	\begin{equation} \label{eq:main_control}	
	\|u\|_{L_{\omega}^{\rho_{0}}\xX(0,T_0)} \leq C \|u_{0}\|_{L_{\omega}^{\rho_{0}}L_{x}^{2}}. 
	\end{equation}
	Furthermore, $\omega$-almost surely, one has the mass conservation $\|u(t)\|_{L_x^2} = \|u_0\|_{L_x^2}$ for all $t \in [0,T_0]$. 
\end{thm}

Since the short existence time $T_0$ above depends on $\delta_0$ only and is \textit{deterministic} , and we have the pathwise mass conservation, Theorem~\ref{th:main} immediately implies global existence of the solution.

\begin{cor}
There exists a  $\delta_{0}>0$  such that for every initial data $u_0$ with $\|u_{0}\|_{L_{\omega}^{\infty}L_{x}^{2}} \leq \delta_{0}$, the solution $u$ in Theorem~\ref{th:main} can be extended globally in time. 
\end{cor}

\begin{rmk}
Our solution $u$ to \eqref{eq:main_eq} is constructed via a unique fixed approximation procedure rather than a direct contraction principle. Thus, the uniqueness in Theorem~\ref{th:main} is in the quasilinear sense rather than semi-linear sense. In short, given every small initial data $u_0$, our construction produces a unique global flow $u$ satisfying \eqref{eq:duhamel_main}. Our solution should not be confused with the weak solutions which are typically obtained by compactness arguments. Also, even though Theorem~\ref{th:main} states that $u \in L_{\omega}^{\rho_0} \xX(0,T_0)$, the same $u$ automatically lives in $L_{\omega}^{\rho} \xX(0,T_0)$ for every $\rho \in (\rho_0, +\infty)$. 
\end{rmk}

\begin{rmk}
	Although we state the equation in the defocusing case, we did not use the sign in front of the nonlinearity at all. In fact, one can multiply any $\lambda \in \RR$ in front of the nonlinearity, and all the results in this article still hold without modification of the arguments. We choose to set $\lambda = 1$ just for simplicity of the presentation. 
\end{rmk}

\subsection{Background}
The nonlinear Schr\"odinger equation naturally arises from various physics models, for example , many body quantum system, nonlinear optics. The aim of the article is to investigate the impact of a multiplicative noise to the dynamics of mass critical NLS.

The local well-posedness of the deterministic (defocusing) nonlinear Schr\"odinger equation 
\begin{equation} \label{eq:deterministic}
i\partial_t u+\Delta u=|u|^{p} u, \qquad u_0 \in L_{x}^{2}(\RR^d)
\end{equation}
for $p \in [0, \frac{4}{d}]$ is based on Strichartz estimates and is standard nowadays. In short, every $L_{x}^{2}$ initial data gives rise to a space-time function $u$ that satisfies \eqref{eq:deterministic} locally in time. We refer to \cite{cazenave1989some}, \cite{cazenave2003semilinear} and \cite{tao2006nonlinear} for more details. When in the mass subcritical case for general $L_{x}^{2}$ data or mass critical case for small $L^{2}$ data, the local existence time depends on the size of the data only, and one can extend the solution globally in time thanks to the conservation law. The general $L^{2}$ data problem for the mass critical case is much harder, but finally resolved in a series of works by Dodson( \cite{dodson2012global, dodson2016global, dodson2467global}). 

We remark that the behavior of mass critical problem and mass subcritical problem are different. When $p=\frac{4}{d}$, the nonlinear part of the equation and linear part of the equation morally have some strength. This will make the problem more subtle.

The study of local well-posedness for mass subcritical (defocusing) stochastic nonlinear Schr\"{o}dinger equation with a conservative multiplicative noise for $L^{2}$ initial data has been initiated in \cite{BD}. Global theory follows from local theory via pathwise mass conservation.  We want to remark here that even the local theory in \cite{BD} depends on the mass conservation law, and in particular is not totally of perturbative nature, which is very different from the deterministic case. There have been subsequent works in various refinements in the stochastic subcritical cases (see for example \cite{Rockner1} and \cite{Rockner2} for extensions to non-conservative cases). 

In fact, even in the subcritical case (with $|u|^4$ replaced by $|u|^{4-\eps}$ in \eqref{eq:main_eq}), if one tries to directly construct a solution of \eqref{eq:duhamel_main} via contraction map (say in $L_{\omega}^{\rho_0} \xX$ for some $\rho_0$), then one may wonder why \eqref{eq:duhamel_main} is possible to hold since the integrability of the nonlinearity in the probability space can only be in $L^{\rho_{0}/5}$ rather than $L^{\rho_{0}}$. The key point is that, as in the work of \cite{BD}, the construction relies on the pathwise mass conversation law, and is not totally perturbative, so the stochastic process constructed there satisfies \eqref{eq:main_control} for all $\rho_{0}<\rho<\infty$. We also emphasise here that we need the initial data to be small in $L^{\infty}_{\omega}L_{x}^{2}$. It is not clear at this moment whether and how our method can deal with the case when initial data is small in $L_{\omega}^{\rho} L_{x}^2$ for some $\rho < +\infty$. 

In \cite{Hornung}, the author constructed a local solution to the critical equation \eqref{eq:main_eq} stopped at a time when the $L_t^5 L_x^{10}$ norm of the solution reaches some small positive value. This type of local well-posedness cannot be directly combined with the mass conservation law to give a solution global in time.

\subsection{Difficulties in adapting the deterministic theory}
\label{sec:difficulty}

The classical deterministic theory for the local well-posedness of nonlinear Schr\"{o}dinger equation follows form Picard iteration regime and is of perturbative nature. Let us consider \eqref{eq:deterministic} in $d=1$. To construct a local solution, one needs to show the operator $\Gamma$ given by
\begin{equation}
(\Gamma v)(t):=e^{it\Delta}u_{0}+i\int_{0}^{t}\sS(t-s)\big(|v(s)|^{p}v(s) \big) {\rm d} s
\end{equation}
defines a contraction in a suitable function space. Such spaces are indicated by the Strichartz estimates \eqref{eq:Strichartz_1} and \eqref{eq:Strichartz_2} below, for example
\begin{equation} \label{eq:deterministic_space}
\Big\{ v \in \xX(0,T): \|v\|_{\xX_2(0,T)} \leq M \Big\}
\end{equation}
for some suitable $T$ and $M$, where $\xX$ and $\xX_2$ are as in \eqref{eq:space}. In the subcritical case ($p<4$), whatever $M$ is, we can always choose $T$ small enough so that $\Gamma$ forms a contraction in the space \eqref{eq:deterministic_space}. This is because the Strichartz estimate would give us the factor $T^{1-\frac{p}{4}}$ in front of the nonlinearity (see Proposition~\ref{pr:pre_nonlinear} below). This factor is not available in the critical case when $p=4$. Nevertheless, one can still choose $T$ small enough so that $\|e^{it\Delta}u_0\|_{L_{t}^{5}L_{x}^{10}(0,T)}$ is small, and then $M = 2 \|e^{it\Delta}u_0\|_{L_{t}^{5}L_{x}^{10}(0,T)}$, which is also small. This again ensures $\Gamma$ to be a contraction in the above space. 

We now turn to the stochastic problem
\begin{equation*}
i \partial_t u + \Delta u = |u|^{p} u + u \circ \frac{{\rm d} W}{{\rm d} t}. 
\end{equation*}
The natural space to search for the solution is $L_{\omega}^{\rho} \xX(0,T)$ for some $T>0$ and $\rho \geq 1$. However, the first obstacle is that the analogous Duhamel operator in this case does not even map $L_{\omega}^{\rho}\xX(0,T)$ to itself, whatever $\rho$ and $T$ are! This is because if $v \in L_{\omega}^{\rho}$, then the nonlinearity can only be in $L_{\omega}^{\rho/(p+1)}$. 

To overcome this problem, \cite{BD} introduced a truncation to kill the nonlinearity whenever $\|u\|_{\xX_{2}(0,t)}$ reaches $m$. In the fixed point problem, one then can replace $p$ powers of $u$ by $m^p$, thus yielding an operator mapping a space to itself (with the size of the norms depending on $m$). In the subcritical case $p<4$, one can make the time $T = T_m$ small enough to get a contraction, so that one gets a local solution in $u_{m} \in L_{\omega}^{\rho}(0,T_m)$. It can be extended globally due to pathwise mass conservation. Finally, they showed that the sequence of solutions $\{u_{m}\}$ actually converges as $m \rightarrow +\infty$. This relies on a uniform bound on $\{u_{m}\}$ when $p<4$. 

However, this construction does not extend to the critical case $p=4$, as there is no positive power of $T$ available to compensate the largeness of $m^4$ (unless $m$ itself is very small, in which case does not relate to the original problem any more).

\subsection{Overview of construction}
\label{sec:strategy}

The starting point of our construction is the existence of the solution to the truncated subcritical problem from \cite{BD}. We then obtain uniform bounds on these solutions that allow us to reach criticality (sending $p \rightarrow 4$) and remove the truncation (sending $m \rightarrow +\infty$). 

To be precise, we let $\theta: \RR \rightarrow \RR^{+}$ be a smooth function with compact support in $(-2,2)$ and $\theta = 1$ on $[-1,1]$. For every $m>0$, let $\theta_{m}(x) = \theta(x/m)$. Consider the truncated sub-critical equation
\begin{equation} \label{eq:trun_sub}
i \partial_t u_{m,\eps}+\Delta u_{m,\eps} = \theta_{m} \big(\|u_{m,\eps}\|_{\xX_{2}(0,t)} \big) \nN^{\eps}(\uem) + u_{m,\eps} \circ \frac{{\rm d} W}{{\rm d} t}, \quad u_0 \in L_{\omega}^{\infty}L_{x}^{2}, 
\end{equation}
where $\nN^{\eps}(u) = |u|^{4-\eps}u$. Here and throughout the rest of the article, we use $\eps$ to denote $4-p$. The following global existence theorem is contained in \cite{BD}. 

\begin{thm}\label{th:global_trun}
	Let $\|u_0\|_{L_{\omega}^{\infty}L_{x}^{2}} < +\infty$. For every $m,\eps>0$ and every sufficiently large $\rho$, there exists $T = T(m, \eps, \rho, \|u_0\|_{L_{\omega}^{\infty}L_{x}^{2}})$ such that the equation \eqref{eq:trun_sub} has a unique solution $\uem$ in $L_{\omega}^{\rho}\xX(0,T)$. It satisfies the Duhamel formula
	\begin{equation} \label{eq:duhamel_trun_sub}
	\begin{split}
	u_{m,\eps}(t) &= \sS(t) u_0 - i \int_{0}^{t} \sS(t-s) \Big( \theta_{m} \big(\|\uem\|_{\xX_{2}(0,s)} \big) \nN^{\eps}(\uem(s))\Big) {\rm d} s\\
	&- i \int_{0}^{t} \sS(t-s) \uem(s) {\rm d} W_s - \frac{1}{2} \int_{0}^{t} \sS(t-s) \big(F_{\Phi} \uem(s) \big) {\rm d} s
	\end{split}
	\end{equation}
	in $L_{\omega}^{\rho} \xX(0,T)$. Furthermore, we have pathwise mass conservation in the sense that  
	\begin{equation}
	\|u_{m,\eps}(t)\|_{L_{x}^{2}} = \|u_0\|_{L_{x}^{2}}, \qquad \forall t \in [0,T]
	\end{equation}
	almost surely. As a consequence, one can iterate the construction to get a global in time flow $\uem$.
\end{thm}

\begin{remark}
	Strictly speaking, due to the truncation in the nonlinearity, the equation \eqref{eq:trun_sub} is \textit{not} translation invariant in time. Thus, it is not immediately obvious that a local theory with mass conservation can imply global existence. However, the nonlinearity satisfies a bound of the form
	\begin{equation*}
	\Big\|\int_{0}^{t}\theta_{m} \big(\|\uem\|_{\xX_{2}(0,s)} \big) \nN^{\eps}\big(\uem(s)\big) {\rm d}s \Big\|_{\xX(0,T_0)} \leq C T_{0}^{\frac{\eps}{4}} m^{4}\|\uem\|_{\xX(0,T_0)}. 
	\end{equation*}
	The right hand side above is translation invariant in time, and hence one can iterate the local construction to get a global flow.
\end{remark}

Theorem \ref{th:global_trun} follows from standard Picard iteration (which is purely perturbative) and relies essentially on Strichartz estimates and Burkholder inequality. We refer to \cite[Proposition 3.1]{BD} for more details of the proof. The main result of \cite{BD}, however, is a uniform bound for the sequence $\{u_{m,\eps}\}_{m}$ for every fixed $\eps>0$. This allows the authors to show the convergence of $u_{m,\eps}$ to a limit $u_{\infty,\eps}$, and that the limit solves the corresponding subcritical equation without truncation. 

On the other hand, in order to obtain a solution for the critical problem \eqref{eq:main_eq}, we need to send $\eps \rightarrow 0$ and $m \rightarrow +\infty$. Instead of starting from the subcritical solution $u_{\infty,\eps}$ directly, we fix $m$ and send $\eps \rightarrow 0$ first. 

\begin{prop}\label{pr:converge_e}
	Let $\uem$ be the solution to \eqref{eq:trun_sub} with initial data $u_0$, so that it satisfies the Duhamel's formula \eqref{eq:duhamel_trun_sub}. There exists $\delta_0>0$ and $T_0$ depending on $\delta_0$ only (independent of $m$ and $\eps$) such that if
	\begin{equation*}
	\|u_0\|_{L_{\omega}^{\infty}L_{x}^{2}} \leq \delta_0, 
	\end{equation*}
	then for every $m>0$ and $\rho_0 \geq 5$, the sequence $\{u_{m,\eps}\}$ converges to $u_{m}$ in $L_{\omega}^{\rho_0}\xX(0,T_0)$ as $\eps \rightarrow 0$. Furthermore, the limit $u_{m}$ satisfies the Duhamels' formula
	\begin{equation} \label{eq:trun_critical}
	\begin{split}
	u_{m}(t) &= \sS(t) u_0 - i \int_{0}^{t} \sS(t-s) \Big( \theta_{m}\big(\|u_{m}\|_{\xX_{2}(0,t)}\big) \nN\big(u_{m}(s)\big)\Big) {\rm d}s\\
	&- i \int_{0}^{t} \sS(t-s) u_{m}(s) {\rm d} W_s - \frac{1}{2} \int_{0}^{t} \sS(t-s) \big( F_{\Phi} u_{m}(s) \big) {\rm d}s
	\end{split}
	\end{equation}
	in $L_{\omega}^{\rho_0}\xX(0,T_0)$. 
\end{prop}

Our next step is to show that the sequence $\{u_{m}\}$ also converges to a limit in $L_{\omega}^{\rho_0} \xX(0,T_0)$. This requires the following uniform bound. 

\begin{prop} \label{pr:uniform_e_m}
	Let $u_{m,\eps}$ be the solution to \eqref{eq:trun_sub}. There exists $\delta_0 > 0$ and $T_0>0$ depending on $\delta_0$ only such that if
	\begin{equation*}
	\|u_0\|_{L_{\omega}^{\infty} L_{x}^{2}} \leq \delta_{0}, 
	\end{equation*}
	then for every $\rho \geq 5$, there exists $C=C_{\rho}$ such that
	\begin{equation*}
	\|\uem\|_{L_{\omega}^{\rho}\xX(0,T_0)} \leq C \|u_0\|_{L_{\omega}^{\rho}L_{x}^{2}}. 
	\end{equation*}
	The constant $C$ depends on $\rho$ but is independent of $\eps$ and $m$. 
\end{prop}

In order to show the convergence of $\{u_{m}\}$ in $L_{\omega}^{\rho_0}\xX(0,T_0)$ as $m \rightarrow +\infty$, it is essential that the uniform bound above holds with a strictly larger $\rho$ (we need $\rho>4\rho_0$ in our case). Now, with the help of Proposition~\ref{pr:uniform_e_m}, we can prove that $\{u_{m}\}$ is Cauchy in $L_{\omega}^{\rho_0}\xX(0,T_0)$, and the limit $u$ satisfies the Duhamel formula \eqref{eq:duhamel_main} in the same space. This concludes the proof of Theorem~\ref{th:main}. 

\begin{rmk}
	The proof of Proposition~\ref{pr:converge_e} does not require the initial data to be small. In fact, if we allow the bounds to depend on $m$, then we can start with any initial data with finite $L_{\omega}^{\infty} L_{x}^{2}$ norm. It is the uniform (in both $\eps$ and $m$) bound in Proposition~\ref{pr:uniform_e_m} and convergence of $\{u_{m}\}$ that rely crucially on the smallness of $\delta_0$. 
	
	On the other hand, in order to make the assumptions consistent in this article, we assume smallness of $\delta_0$ throughout this article, including the part for Proposition~\ref{pr:converge_e}. Its proof will also be simplified with the help of Proposition~\ref{pr:uniform_e_m}. Hence, in what follows, we will prove the uniform bound in Proposition~\ref{pr:uniform_e_m} first, and then turn to the convergence of $\{u_{m,\eps}\}_{\eps>0}$ for fixed $m$. 
\end{rmk}

\begin{rmk}
	As mentioned earlier, the key point in \cite{BD} is  fixing $\eps$, and show uniform boundedness for $\uem$ in $m$. We remark that they can handle general $L^{2}$ data without smallness requirement. 
\end{rmk}

\begin{rmk}
	We finally remark that it might be possible to use rough path theory and regularity structures developed in \cite{rp, controlled_rp, rs} to develop a pathwise solution theory to \eqref{eq:main_eq}. This would avoid the problem of the stochastic integrability, and may give a direct construction via a fixed point argument. We plan to investigate this issue in future work.
\end{rmk}

\subsection*{Structure of the article}

The rest of the article is organised as follows. We first give some preliminary lemmas in Section~\ref{sec:preliminary}. In Section~\ref{sec:uniform}, we give a few bounds that will be used throughout the rest of the article, and also prove the uniform boundedness of the truncated subcritical solutions. Section~\ref{sec:converge_e} is the most technical part of the article. It gives the convergence in of subcritical solutions to the critical one with the truncation present. Finally, in Section~\ref{sec:converge_m}, we show that the truncation can be removed, thus yielding a construction of the solution to \eqref{eq:main_eq} and obtaining its stability.

\subsection*{Notations}

We now introduce the notations used in this article. For any interval $\iI$, we use $L_{t}^{q} L_{x}^{r} (a,b)$ to denote the space $L^{q}(\iI, L^{r}(\RR))$, and we also write $L_{\omega}^{\rho} \yY = L^{\rho}(\Omega, \yY)$. We fix the spaces $\xX_1$ and $\xX_2$ to be
\begin{equation*}
\xX_1 (\iI) = L_{t}^{\infty} L_{x}^{2}(\iI), \qquad \xX_{2}(\iI) = L_{t}^{5} L_{x}^{10}(\iI), 
\end{equation*}
and $\xX(\iI) = \xX_{1}(\iI) \cap \xX_{2}(\iI)$ with the norm being their sum. Some intermediate steps in the proof require us to go to a higher regularity space than $L_{x}^{2}$, so we let $\xX^{1}(\iI)$ be the space of functions such that
\begin{equation*}
\|u\|_{\xX^{1}(\iI)} \eqdef \|u\|_{\xX(\iI)} + \|\d_x u\|_{\xX(\iI)} < +\infty. 
\end{equation*}
Throughout this article, we fix a universal small number $\delta_0 > 0$ and $T_0 > 0$ depending on $\delta_0$ so that Proposition~\ref{pr:uniform_e_m} holds. The smallness of $\delta_0$ will be specified later. We also fix an arbitrary $\rho_0 > 5$. Since these values are fixed, we will omit the dependence of the estimates below on them. 

We also write $\nN(u) = |u|^{4} u$ and $\nN^{\eps}(u) = |u|^{4-\eps} u$. Finally, we fix $\theta$ to be a non-negative smooth function on $\RR$ with compact support in $(-2,2)$ such that $\theta(x) = 1$ for $|x| \leq 1$. For every $m>0$, we let $\theta_{m}(x) = \theta(x/m)$.

\subsection*{Acknowledgements}
We thank Carlos Kenig and Gigliola Staffilani for helpful discussions and comments. This work was initiated during the Fall 2015 program ``New Challenges in PDE: Deterministic
Dynamics and Randomness in High and Infinite Dimensional Systems" held at MSIR, Berkeley. We thank for MSRI for providing a stimulating mathematical environment. 

Part of this work was done when CF was a graduate student in MIT. CF was partially supported by NSF DMS 1362509 and DMS 1462401. CF would also like to thank Chuntian Wang and Jie Zhong for discussions. WX acknowledges the support from the Engineering and Physical Sciences Research Council through the fellowship EP/N021568/1. 

\section{Preliminaries}\label{sec:preliminary}

\subsection{The Wiener process and Burkholder inequality}

Our main assumption on the noise $W$ is that it can be written as $W = \Phi \tilde{W}$, where $\tilde{W}$ is the cylindrical Wiener process on $L^2(\RR)$, and $\Phi$ is a trace-class operator satisfying Assumption~\ref{as:Phi}. 

We now introduce the notion of $\gamma$-radonifying operators since the Burkholder inequality we use below is most conveniently expressed with this notion. A linear operator $\Gamma: \bB \rightarrow \tilde{\hH}$ from a Banach space $\bB$ to a Hilbert space $\tilde{\hH}$ is $\gamma$-radonifying if for any sequence $\{\gamma_k\}_k$ of independent standard normal random variables on a probability space $(\tilde{\Omega}, \tilde{\fF}, \tilde{\PP})$ and any orthonormal basis $\{e_k\}_k$ of $\tilde{\hH}$, the series $\sum_k \gamma_k \Gamma e_k$ converges in $L^2(\tilde{\Omega}, \bB)$. The $\gamma$-radonifying norm of the operator $\Gamma$ is then defined by
\begin{equation*}
\|\Gamma\|_{\R(\bB, \tilde{H})} = \Big( \EE \|\sum_k \gamma_k \Gamma e_k\|_{\bB}^{2} \Big)^{\frac{1}{2}}, 
\end{equation*}
which is independent of the choice of $\{\gamma_k\}$ or $\{e_k\}$. We then have
\begin{equation*}
\|\Phi\|_{\R(L_{x}^{2}, \hH)} < +\infty
\end{equation*}
for the Hilbert space $\hH$ specified in Assumption~\ref{as:Phi}. We also need the following factorisation lemma. 

\begin{lem} \label{le:factorisation}
	Let $\kK$ be a Hilbert space, and $\eE$ and $\bB$ be Banach spaces. For every $\Gamma \in \R(\kK,\eE)$ and $\tT \in \LLL(\eE,\bB)$, we have $\tT \circ \Gamma \in \R(\kK,\bB)$ with the bound
	\begin{equation*}
	\|\tT \circ \Gamma\|_{\R(\kK,\bB)} \leq \|\tT\|_{\LLL(\eE,\bB)} \|\Gamma\|_{\R(\kK,\eE)}. 
	\end{equation*}
	In particular, if $\eE = L^p$ and $\tT$ is given by the multiplication of an $L^q$ function $\sigma$ with $\frac{1}{p} + \frac{1}{q} = \frac{1}{r} \leq 1$, then $\sigma \Gamma \in \R(\kK,L^r)$ with
	\begin{equation*}
	\|\sigma \Gamma\|_{\R(\kK,L^r)} \leq \|\sigma\|_{L^q} \|\Gamma\|_{\R(\kK,L^p)}. 
	\end{equation*}
\end{lem}
\begin{proof}
	The first claim is same as \cite[Lemma 2.1]{BD}. The second claim then follows immediately from the first one and H\"{o}lder's inequality. 
\end{proof}

The Burkholder inequality (\cite{BDG, Burkholder}) is very useful in controlling moments of the supremum of a martingale. We will make use of the following version. 

\begin{prop} \label{pr:Burkholder}
	Let $W = \Phi \tilde{W}$ with $\tilde{W}$ and $\Phi$ be as described above. Let $\sigma$ be adapted to $\fF_t$. Then, for every $p \in [2,\infty)$ and every $\rho \in [1,\infty)$, we have
	\begin{equation*}
	\EE \sup_{t \in [a,b]} \Big\| \int_{a}^{t} \sigma(s) {\rm d} W_s \Big\|_{L^p}^{\rho} \leq C_{p,\rho} \EE \Big( \int_{0}^{T} \|\sigma(s) \Phi\|_{\R(L^{2}, L^{p})}^{2} {\rm d} s \Big)^{\frac{\rho}{2}}. 
	\end{equation*}
	Here, the operator $\sigma(s) \Phi$ is such that $\sigma(s) \Phi f = \sigma(s) \cdot \Phi f$. 
\end{prop}

The proof can be found, for example, in \cite[Theorem 2.1]{BP}. More details about this version of the inequality can be found in \cite{Brzezniak, UMD}.

\subsection{Dispersive and Strichartz estimates}

We give some dispersive and Strichartz estimates that will be used throughout this article. These are standard now, and can be found in \cite{cazenave2003semilinear}, \cite{keel1998endpoint} and \cite{tao2006nonlinear}. 

Recall that $\sS(t) = e^{i t \Delta}$. We need the following dispersive estimates of the semi-group $\sS$ in $d=1$. 

\begin{prop} \label{pr:dispersive}
	There exists a universal constant $C>0$ such that
	\begin{equation} \label{eq:dispersive}
	\|\sS(t) f\|_{L^{p'}} \leq C t^{\frac{1}{p}-\frac{1}{2}} \|f\|_{L^{p}}
	\end{equation}
	for every $p \in [1,2]$ and every $f \in L^{p}(\RR)$. Here, $p'$ is the conjugate of $p$. 
\end{prop}

We now turn to Strichartz estimates. A pair of real numbers $(q,r)$ is called an admissible pair (for $d=1$) if
\begin{equation} \label{eq:admissible}
\frac{2}{q} + \frac{1}{r} = \frac{1}{2}. 
\end{equation}
The following Strichartz estimates gives the right space to look for solutions. 

\begin{prop} \label{pr:Strichartz}
	For every two admissible pairs $(q,r)$ and $(\tilde{q}, \tilde{r})$, we have
	\begin{equation} \label{eq:Strichartz_1}
	\|\sS(t) f\|_{L_{t}^{q}L_{x}^{r}} \lesssim \|f\|_{L_{x}^{2}}, 
	\end{equation}
	and
	\begin{equation} \label{eq:Strichartz_2}
	\Big\|\int_{\tau}^{t} \sS(t-s) f(s) {\rm d} s \Big\|_{L_{t}^{q}L_{x}^{r}} \lesssim \|f\|_{L_{t}^{\tilde{q}'}L_{x}^{\tilde{r}'}}, 
	\end{equation}
	where $\tilde{q}'$, $\tilde{r}'$ are conjugates of $\tilde{q}$ and $r$. The proportionality constants depend on these parameters but are independent of $f$ and the length of the time interval. 
\end{prop}

\section{Uniform boundedness} 
\label{sec:uniform}

\subsection{Some preliminary bounds}
\label{sec:bounds}

We now give some bounds that will be used in the proof of Proposition~\ref{pr:uniform_e_m} as well as in later sections. Throughout the rest of the article, we use $\sigma$ to denote processes satisfying different assumptions in various statements. But we should think of them as the solution to the truncated subcritical problem \eqref{eq:trun_sub}, or the difference between its two solutions starting from different initial data. 

Also recall the notations $\xX_{1}(\iI) = L_{t}^{\infty}L_{x}^{2}(\iI)$, $\xX_{2}(\iI) = L_{t}^{5} L_{x}^{10}(\iI)$, and $\xX = \xX_1 \cap \xX_2$. We will always use $\iI$ to denote the interval $[a,b]$ concerned in the contexts below. 

\begin{prop} \label{pr:pre_nonlinear}
	Recall that $\nN^{\eps}(\sigma) = |\sigma|^{4-\eps} \sigma$. We then have the bound
	\begin{equation*}
	\Big\| \int_{a}^{t} \sS(t-s) \nN^{\eps} \big( \sigma(s) \big) {\rm d}s \Big\|_{\xX(\iI)} \leq C (b-a)^{\frac{\eps}{4}} \|\sigma\|_{\xX_{1}(\iI)} \|\sigma\|_{\xX_{2}(\iI)}^{4-\eps}. 
	\end{equation*}
\end{prop}
\begin{proof}
	By Strichartz estimate \eqref{eq:Strichartz_2}, we have
	\begin{equation*}
	\Big\| \int_{a}^{t} \sS(t-s) \nN^{\eps} \big( \sigma(s) \big) {\rm d}s \Big\|_{\xX(\iI)} \leq C \|N^{\eps}(\sigma)\|_{L_{t}^{q_{\eps}'}L_{x}^{r_{\eps}'}(\iI)}, \quad q_{\eps}' = \frac{20}{16 + \eps},\; r_{\eps}' = \frac{10}{9-\eps}. 
	\end{equation*}
	It is straightforward to check that their conjugates $(q_{\eps}, r_{\eps})$ satisfy the admissibility condition \eqref{eq:admissible}. It then remains to control the quantity
	\begin{equation*}
	\|\nN^{\eps}(\sigma)\|_{L_{t}^{q_{\eps}'}L_{x}^{r_{\eps}'}(\iI)} = \Big( \int_{a}^{b}  \|\nN^{\eps} \big( \sigma(t) \big)\|_{L_{x}^{r_{\eps}'}}^{q_{\eps}'} {\rm d}t \Big)^{\frac{1}{q_{\eps}'}}. 
	\end{equation*}
	A direct application of H\"{o}lder gives
	\begin{equation*}
	\|\nN^{\eps} \big( \sigma(t) \big)\|_{L_{x}^{r_{\eps}'}} \leq \|\sigma(t)\|_{L_{x}^{2}} \|\sigma(t)\|_{L_{x}^{10}}^{4-\eps} \leq \|\sigma\|_{\xX_{1}(\iI)} \|\sigma(t)\|_{L_{x}^{10}}^{4-\eps}. 
	\end{equation*}
	Another application of H\"{o}lder to the integration in time gives the desired bound. 
\end{proof}

In the subcritical situation with $\eps>0$, the factor $(b-a)^{\frac{\eps}{4}}$ is crucial for constructing the local solution via contraction. However, to show the convergence of solutions to subcritical problems, we need to get uniform in $\eps$ estimates. Hence, this factor will be of little use to us. In what follows, we will always use the following bound. 

\begin{cor}
	There exists $C>0$ such that
	\begin{equation*}
	\Big\| \int_{a}^{t} \sS(t-s) \nN^{\eps} \big( \sigma(s) \big) {\rm d}s \Big\|_{\xX(\iI)} \leq C \|\sigma\|_{\xX_{1}(\iI)} \|\sigma\|_{\xX_{2}(\iI)}^{4-\eps}
	\end{equation*}
	for all $[a,b] \subset [0,T_0]$ and every $\eps > 0$. 
\end{cor}

Let $\sigma$ be a process adapted to the filtration $\fF_{t}$. The aim of this subsection is to give controls on norms related to the stochastic quantity
\begin{equation*}
\int_{a}^{t} \sS(t-s) \sigma(s) {\rm d} W_s, 
\end{equation*}
where the integral is in the It\^{o} sense. Let
\begin{equation*}
\begin{split}
M_{1,a}^{*}(t) &= \sup_{a \leq r_{1} \leq r_{2} \leq t} \Big\| \int_{r_1}^{r_2} \sS(t-s) \sigma(s) {\rm d} W_s \Big\|_{L_{x}^{2}},\\
M_{2,a}^{*}(t) &= \sup_{a \leq r_{1} \leq r_{2} \leq t} \Big\| \int_{r_1}^{r_2} \sS(t-s) \sigma(s) {\rm d} W_s \Big\|_{L_{x}^{10}}. 
\end{split}
\end{equation*}
We have the following proposition. 

\begin{prop} \label{pr:pre_stochastic}
	For every $\rho \geq 5$, we have the bounds
	\begin{equation*}
	\begin{split}
	\|M_{1,a}^{*}\|_{L_{\omega}^{\rho}L_{t}^{\infty}(\iI)} &\leq C_{\rho} (b-a)^{\frac{1}{2}} \|\Phi\|_{\R(L_{x}^{2}, L_{x}^{\infty})} \|\sigma\|_{L_{\omega}^{\rho} \xX_{1}(\iI)};\\
	\|M_{2,a}^{*}\|_{L_{\omega}^{\rho}L_{t}^{5}(\iI)} &\leq C_{\rho} (b-a)^{\frac{3}{10}} \|\Phi\|_{\R(L_{x}^{2}, L_{x}^{5/2})} \|\sigma\|_{L_{\omega}^{\rho} \xX_{1}(\iI)}, 
	\end{split}
	\end{equation*}
	where the proportionality constants depend on $\rho$ only. 
\end{prop}
\begin{proof}
	We first treat $M_{1,a}^{*}$. Since $\sS$ is unitary, we have
	\begin{equation*}
	M_{1,a}^{*}(t) = \sup_{a \leq r_{1} \leq r_{2} \leq t} \Big\| \int_{r_1}^{r_2} \sS(-s) \sigma(s) {\rm d} W_{s} \Big\|_{L_{x}^{2}} \leq 2 \sup_{a \leq r \leq b} \Big\| \int_{a}^{r} \sS(-s) \sigma(s) {\rm d} W_{s} \Big\|_{L_{x}^{2}}. 
	\end{equation*}
	The right hand side above does not depend on $t$, so we have
	\begin{equation*}
	\|M_{1,a}^{*}\|_{L_{t}^{\infty}(\iI)} \leq 2 \sup_{a \leq r \leq b} \Big\| \int_{a}^{r} \sS(-s) \sigma(s) {\rm d} W_{s} \Big\|_{L_{x}^{2}}. 
	\end{equation*}
	Now, since the process $s \mapsto \sS(-s) \sigma(s)$ is adapted to $\FFF_t$, we can apply Burkholder inequality in Proposition~\ref{pr:Burkholder} (with $\sS(-s) \sigma(s)$ replacing $\sigma(s)$) to get
	\begin{equation*}
	\EE \|M_{1,a}^{*}\|_{L_{t}^{\infty}(\iI)}^{\rho} \lesssim_{\rho} \EE \Big( \int_{a}^{b} \|\sS(-s) \sigma(s) \Phi\|_{\R(L_{x}^{2}, L_{x}^{2})}^{2} {\rm d}s \Big)^{\frac{\rho}{2}}. 
	\end{equation*}
	Using again the unitary property of $\sS(-s)$ and Lemma \ref{le:factorisation}, we have
	\begin{equation*}
	\|\sS(-s) \sigma(s) \Phi\|_{\R(L_{x}^{2}, L_{x}^{2})} \leq \|\sigma(s)\|_{L_{x}^{2}} \|\Phi\|_{\R(L_{x}^{2}, L_{x}^{\infty})}. 
	\end{equation*}
	Plugging it back to the above bound for $\EE \|M_{1,a}^{*}\|_{L_{t}^{\infty}}^{\rho}$ and applying H\"{o}lder, we get
	\begin{equation*}
	\EE \|M_{1,a}^{*}\|_{L_{t}^{\infty}(\iI)}^{\rho} \lesssim_{\rho} (b-a)^{\frac{\rho}{2}} \|\Phi\|_{\R(L_{x}^{2}, L_{x}^{\infty})}^{\rho} \EE \|\sigma\|_{\xX_{1}(\iI)}^{\rho}. 
	\end{equation*}
	Taking $\rho$-th root on both sides gives the desired bound for $M_{1,a}^{*}$. As for $M_{2,a}^{*}$, since $\rho \geq 5$, we can use Minkowski to change the order of integration and then apply H\"{o}lder so that
	\begin{equation} \label{eq:minkowski_exchange}
	\|M_{2,a}^{*}\|_{L_{\omega}^{\rho} L_{t}^{5}(\iI)} \leq \|M_{2,a}^{*}\|_{L_{t}^{5}(\iI, L_{\omega}^{\rho})} \leq (b-a)^{\frac{1}{5}} \sup_{t \in [a,b]} \|M_{2,a}^{*}(t)\|_{L_{\omega}^{\rho}}. 
	\end{equation}
	Since
	\begin{equation*}
	M_{2,a}^{*}(t) \leq 2 \sup_{0 \leq \tau \leq t} \Big\| \int_{a}^{\tau} \sS(t-s) \sigma(s) {\rm d} W_{s} \Big\|_{L_{x}^{10}}, 
	\end{equation*}
	we use Burkholder inequality to get
	\begin{equation*}
	\|M_{2,a}^{*}(t)\|_{L_{\omega}^{\rho}}^{\rho} \lesssim_{\rho} \EE \Big( \int_{a}^{t} \|\sS(t-s) \sigma(s) \Phi\|_{\R(L_{x}^{2}, L_{x}^{10})}^{2} {\rm d}s \Big)^{\frac{\rho}{2}}. 
	\end{equation*}
	Now, applying the dispersive estimate \eqref{eq:dispersive} and Lemma \ref{le:factorisation} to the integrand, we get
	\begin{equation*}
	\|\sS(t-s) \sigma(s) \Phi\|_{\R(L_{x}^{2}, L_{x}^{10})}^{2} \lesssim (t-s)^{-\frac{2}{5}} \|\sigma(s)\|_{L_{x}^{2}} \|\Phi\|_{\R(L_{x}^{2}, L_{x}^{5/2})}. 
	\end{equation*}
	Substituting it back into the bound for $\|M_{2,a}^{*}(t)\|_{L_{\omega}^{\rho}}^{\rho}$, we get
	\begin{equation*}
	\|M_{2,a}^{*}(t)\|_{L_{\omega}^{\rho}}^{\rho} \lesssim_{\rho} (b-a)^{\frac{\rho}{10}} \|\Phi\|_{\R(L_{x}^{2}, L_{x}^{5/2})}^{\rho} \EE \|\sigma\|_{\xX_{1}(\iI)}^{\rho}. 
	\end{equation*}
	Note that the right hand side above does not depend on $t$. So taking the $\rho$-th root on both sides and then supremum over $t \in [a,b]$, and combining it with \eqref{eq:minkowski_exchange}, we obtain the desired control for $M_{2,a}^{*}$. 
\end{proof}

\begin{rmk}
	The exact value of the exponent of $(b-a)$ is not very important. However, it is crucial that the exponent is strictly positive. This will enable us to absorb the term on the right hand side of the equation to the left hand side with a short time period. The same is true for the bounds in Proposition~\ref{pr:pre_correction} below. 
\end{rmk}

\begin{rmk}
	For both of the bounds in Proposition~\ref{pr:pre_stochastic}, one can slightly modify the argument to control the left hand side by the $L_{\omega}^{\rho} \xX_{2}(\iI)$ norm. We choose the $L^2$ based norm $\xX_1$ for convenience of use later. 
\end{rmk}

We now give controls on the quantity
\begin{equation*}
A(t) = \int_{a}^{t} \sS(t-s) \big( F_{\Phi} \sigma(s) \big) {\rm d}s. 
\end{equation*}
We have the following proposition. 

\begin{prop} \label{pr:pre_correction}
	Let $A$ be as above. We have the bounds
	\begin{equation*}
	\begin{split}
	\|A\|_{\xX_{1}(\iI)} &\leq C (b-a) \|F_{\Phi}\|_{L_{x}^{\infty}} \|\sigma\|_{\xX_{1}(\iI)};\\
	\|A\|_{\xX_{2}(\iI)} &\leq C (b-a)^{\frac{4}{5}} \|F_{\Phi}\|_{L_{x}^{5/2}} \|\sigma\|_{\xX_{1}(\iI)}. 
	\end{split}
	\end{equation*}
\end{prop}
\begin{proof}
	We first look at $\xX_{1}$-norm of $A$. By unitary property of $\sS$ and H\"{o}lder, we have
	\begin{equation*}
	\|A(t)\|_{L_{x}^{2}} \leq \int_{a}^{t} \|F_{\Phi} \sigma(s)\|_{L_{x}^{2}} {\rm d}s \leq (b-a) \|F_{\Phi}\|_{L_{x}^{\infty}} \|\sigma\|_{\xX_1(\iI)}. 
	\end{equation*}
	The right hand side above does not depend on $t$, so we have proved the bound for $\|A\|_{\xX_1(\iI)}$. As for the $\xX_{2}$-norm, by dispersive estimate \eqref{eq:dispersive} and H\"{o}lder, we have
	\begin{equation*}
	\|A(t)\|_{L_{x}^{10}} \leq C \int_{a}^{t} (t-s)^{-\frac{2}{5}} \|F_{\Phi} \sigma(s)\|_{L_{x}^{10/9}} {\rm d}s \leq C \|F_{\Phi}\|_{L_{x}^{5/2}} \int_{a}^{t} (t-s)^{-\frac{2}{5}} \|\sigma(s)\|_{L_{x}^{2}} {\rm d}s. 
	\end{equation*}
	Taking the supremum over $s$ for $\|\sigma(s)\|_{L_{x}^{2}}$ and then integrating $s$ out, we get
	\begin{equation*}
	\|A(t)\|_{L_{x}^{10/9}} \leq C (b-a)^{\frac{3}{5}} \|F_{\Phi}\|_{L_{x}^{5/2}} \|\sigma\|_{\xX_{1}(\iI)}.  
	\end{equation*}
	Note that the right hand side above does not depend on $t$. Taking $L_{t}^{5}(\iI)$-norm then immediately gives
	\begin{equation*}
	\|A\|_{\xX_{2}(\iI)} \leq C (b-a)^{\frac{4}{5}} \|F_{\Phi}\|_{L_{x}^{5/2}} \|\sigma\|_{\xX_{1}(\iI)}. 
	\end{equation*}
	This completes the proof of the proposition. 
\end{proof}

\subsection{Proof of Proposition~\ref{pr:uniform_e_m}}
\label{sec:uniform_bound}

We are now ready to prove Proposition \ref{pr:uniform_e_m} in this section. Fix $\rho_0>0$. Let $\delta_0 > 0$ be sufficiently small, and let $u_0$ be such that $\|u_0\|_{L_{\omega}^{\infty}L_x^2} \leq \delta_0$. The smallness of $\delta_0$ will be specified below. 

Let $u_{m,\eps}$ be the solution to \eqref{eq:trun_sub} with initial data $u_0$. For simplicity of notation, we will omit the subscripts $m$ and $\eps$ and write $u = u_{m,\eps}$. By pathwise mass conservation, it suffices to control the $L_{\omega}^{\rho} \xX_2$ norm of $u$. We consider $u$ on the time interval $[0,T_0]$, where $T_0$ depends on $\delta_0$ as follows. By Proposition~\ref{pr:pre_correction}, there exists $C>0$ such that for every $[a,b] \subset [0,T_0]$, we have
\begin{equation} \label{eq:bounded_correction}
\Big\|\int_{a}^{t} \sS(t-s) \big( F_{\Phi} u(s)\big) {\rm d}s \Big\|_{\xX_{2}(a,b)} \leq C (b-a)^{\frac{4}{5}} \|u\|_{\xX_{1}(a,b)} \leq C T_{0}^{\frac{4}{5}} \|u_0\|_{L_x^2}.  
\end{equation}
We choose $T_0$ such that the right hand side above is smaller than $1$. Clearly $T_0$ depends on $\delta_0$ only. Now, for every $\tau$, we let $\qQ_\tau$ and $\gG_\tau$ be the processes on $t \geq \tau$ be
\begin{equation} \label{eq:q_g}
\begin{split}
\qQ_{\tau}(t) &= \int_{\tau}^{t} \sS(t-s) u(s) {\rm d} W_s\;, \\
\gG_{\tau}(t) &= \int_{\tau}^{t} \sS(t-s) \Big( \theta_{m} \big(\|u\|_{\xX_{2}(0,s)} \big) \nN^{\eps} \big(u(s)\big) \Big) {\rm d}s, 
\end{split}
\end{equation}
and we simply write $\qQ$ and $\gG$ for $\qQ_0$ and $\gG_0$. We also let
\begin{equation*}
\qQ^{*}(t) = \sup_{0\leq r_1 < r_2 \leq t} \Big\|\int_{r_1}^{r_2} \sS(t-s) u(s) {\rm d} W_s \Big\|_{L_{x}^{10}}. 
\end{equation*}
Choose a random dissection $\{\tau_k\}$ of the interval $[0,T_0]$ as follows. Let $\tau_0=0$. Suppose $0 = \tau_0 < \tau_1 < \cdots < \tau_k < T_0$ is chosen, we let
\begin{equation} \label{eq:bounded_dissection}
\tau_{k+1} = T_0 \wedge \inf \Big\{ r>\tau_{k}: \int_{\tau_k}^{r} |\qQ^{*}(t)|^{5} dt \geq 1 \Big\}, 
\end{equation}
and we stop this process when it reaches $T_0$. The total number of subintervals in this dissection is at most
\begin{equation} \label{eq:bounded_intervals}
K \leq \max \big\{1,\; 2 \|\qQ^{*}\|_{\xX_{2}(0,T_0)}^{5} \big\}. 
\end{equation}
By Proposition \ref{pr:pre_stochastic}, for every $\rho \geq 5$, there exists $C=C_{\rho}$ such that
\begin{equation} \label{eq:bounded_average}
\|\qQ^{*}\|_{L_{\omega}^{\rho}\xX_{2}(0,T_0)} \leq C_{\rho} (1+T_0) \|u\|_{L_{\omega}^{\rho}\xX_{1}(0,T_0)}. 
\end{equation}
This in particular implies that $K<+\infty$ almost surely. Given this random dissection, almost surely, for every $k$ and every $t \in [\tau_{k}, \tau_{k+1}]$, we have
\begin{equation*}
u(t) = e^{i(t-\tau_k) \Delta} u(\tau_k) - i \gG_{\tau_k}(t) - i \qQ_{\tau_k}(t) - \frac{1}{2} \int_{\tau_k}^{t} \sS(t-s) \big( F_{\Phi} u(s) \big) {\rm d}s, 
\end{equation*}
where $\qQ_{\tau_k}$ and $\gG_{\tau_k}$ are defined as in \eqref{eq:q_g}. For every $r \in [\tau_{k}, \tau_{k+1}]$, we control the $\|\cdot\|_{\xX_{2}(\tau_k,r)}$ norm of the four quantities on the right hand side above separately. The last term has already been treated in \eqref{eq:bounded_correction}. For the first term, by the Strichartz estimate \eqref{eq:Strichartz_1}, we have
\begin{equation} \label{eq:bounded_initial}
\|e^{i (t-\tau_k) \Delta} u(\tau_k)\|_{\xX_{2}(\tau_k,r)} \leq C \|u(\tau_k)\|_{L_{x}^{2}} \leq C \|u_0\|_{L_x^2}. 
\end{equation}
As for the stochastic term $\qQ_{\tau}$, it follows immediately from the choice of the dissection in \eqref{eq:bounded_dissection} that
\begin{equation} \label{eq:bounded_stochastic}
\|\qQ_{\tau_k}\|_{\xX_2(\tau_k,r)}^{5} \leq \int_{\tau_k}^{r} |\qQ^{*}(t)|^{5} dt \leq 1.  
\end{equation}
Finally, for the nonlinear term, we deduce from Proposition~\ref{pr:pre_nonlinear} that
\begin{equation} \label{eq:bounded_nonlinear}
\|\gG_{\tau_k}\|_{\xX_2(\tau_k,r)} \leq C \|u\|_{\xX_{1}(\tau_k,r)} \|u\|_{\xX_{2}(\tau_k,r)}^{4-\eps} \leq C \delta_0 \|u\|_{\xX_{2}(\tau_k,r)}^{4-\eps}. 
\end{equation}
Note that all the bounds above are pathwise. Writing $f_{\tau_k}(r) = \|u\|_{\xX_{2}(\tau_k,r)}$, we have
\begin{equation*}
f_{\tau_k}(r) \leq C_1 (1 + T_{0}^{\frac{4}{5}}) \|u_0\|_{L_{x}^{2}} + \|\qQ_{\tau_k}\|_{\xX_{2}(\tau_k,r)} + C_2 \delta_0 |f_{\tau_k}(r)|^{4-\eps}, 
\end{equation*}
where $\|\qQ_{\tau_k}\|_{\xX_{2}(\tau_k,r)} \leq 1$. This is true for all $r \in [\tau_k, \tau_{k+1}]$. By the choice of $T_0$ (depending on $\delta_0$) in \eqref{eq:bounded_correction} and the bound \eqref{eq:bounded_stochastic}, we have
\begin{equation*}
C_{1} (1 + T_{0}^{\frac{4}{5}}) \|u_0\|_{L_x^2} + \|\qQ_{\tau_k}\|_{\xX_{2}(\tau_k,r)} \leq 3
\end{equation*}
if $\delta_0$ is sufficiently small. Also, since $f_{\tau_k}(\tau_k) = 0$, for small enough $\delta_0$, we can use the standard continuity argument to deduce
\begin{equation} \label{eq:boundedness}
\|u\|_{\xX_{2}(\tau_k,\tau_{k+1})} \leq C \Big( \|u_0\|_{L_{x}^{2}} + \|\qQ_{\tau_k}\|_{\xX_{2}(\tau_k,\tau_{k+1})} \Big), 
\end{equation}
where $C$ is deterministic and depends on $\delta_0$ only. Combining \eqref{eq:bounded_stochastic} and \eqref{eq:bounded_nonlinear}, we see there exists a deterministic $C>0$ such that
\begin{equation} \label{eq:boundedness_0}
\|u\|_{\xX_{2}(\tau_{k}, \tau_{k+1})} \leq C
\end{equation} 
for all $k=0, \cdots, K-1$. If $K=1$, then by \eqref{eq:boundedness}, we have the bound
\begin{equation} \label{eq:boundedness_1}
\|u\|_{\xX_{2}(0,T_0)} \leq C \Big( \|u_0\|_{L_{x}^{2}} +  \|\qQ^{*}\|_{\xX_{2}(0,T_0)} \Big). 
\end{equation}
If $K > 1$, then by \eqref{eq:bounded_intervals} and \eqref{eq:boundedness_0}, we have
\begin{equation*}
\|u\|_{\xX_{2}(0,T_0)} = \Big( \sum_{k=0}^{K-1} \|u\|_{\xX_{2}(\tau_k, \tau_{k+1})} \Big)^{\frac{1}{5}} \leq C \|\qQ^{*}\|_{\xX_{2}(0,T_0)}, 
\end{equation*}
which is also of the form \eqref{eq:boundedness_1}. Hence, there exists a deterministic $C$ such that \eqref{eq:boundedness_1} holds almost surely. We then take $L_{\omega}^{\rho}$ norm on both sides of \eqref{eq:boundedness_1} and use Proposition~\ref{pr:pre_stochastic} to get
\begin{equation*}
\|u\|_{L_{\omega}^{\rho}\xX_{2}(0,T_0)} \leq C_{\rho} \|u\|_{L_{\omega}^{\rho} \xX_{1}(0,T_0)}. 
\end{equation*}
Adding $\|u\|_{L_{\omega}^{\rho} \xX_{1}(0,T_0)}$ to both sides above, we can conclude the claim immediately from mass conservation.

\section{Convergence in $\eps$ -- proof of Proposition~\ref{pr:converge_e}}
\label{sec:converge_e}

The aim of this section is to prove Proposition~\ref{pr:converge_e}. We will show that for every $m$, the sequence of solutions $\{u_{m,\eps}\}_{\eps>0}$ is Cauchy in $L_{\omega}^{\rho_0} \xX(0,T_0)$, and that the limit $u_{m}$ satisfies the corresponding Duhamel's formula. All estimates below are uniform in $\eps$ but depend on $m$. 

\subsection{Overview of the proof}

The proof consists of several ingredients, all of which use bootstrap argument over smaller subintervals of $[0,T_0]$. In order to show $\{u_{m,\eps}\}$ is Cauchy, we can need to control the difference $u_{m,\eps_1} - u_{m,\eps_2}$ for small $\eps_1$ and $\eps_2$ over those subintervals and then iterate. Even though the two processes start with the same initial data, they start to differ instantly after the evolution begins. Hence, in order to be able to iterate over subintervals, we need the solution to be stable under perturbation of initial data. This is the following proposition.

\begin{prop}[Uniform stability in $\eps$] \label{pr:uniform_stable}
	Let $\uem$ and $\vem$ denote the solutions to \eqref{eq:trun_sub} with initial datum $u_0$ and $v_0$ respectively. We have the bound
	\begin{equation*}
	\|u_{m,\eps}-v_{m,\eps}\|_{L_{\omega}^{\rho_0} \xX(0,T)} < C \|u_0 - v_0\|_{L_{\omega}^{\rho_0} L_{x}^{2}}
	\end{equation*}
	for some constant $C$ depending on $m$. 
\end{prop}

The above proposition compares two solutions to the same equation with different initial data. On the other hand, in order to show $\{u_{m,\eps}\}$ is Cauchy in $\eps$, we need to compare $\nN^{\eps_1}(u_{m,\eps_1})$ and $\nN^{\eps_2}(u_{m,\eps_2})$ for $\eps_1 \neq \eps_2$. Since the two nonlinearities carry different powers, we need a priori bound on $L_{x}^{\infty}$ norm of $u_{m,\eps_j}(t)$ to get effective control on their difference. This requires the initial data to be in a more regular space than $L_{x}^{2}$. Hence, we make a small perturbation of initial data to $H_{x}^{1}$. The following proposition guarantees that the solution will still be in $H_{x}^{1}$ up to time $T_0$.

\begin{prop}[Persistence of regularity]
	\label{pr:lwpinhigh}
	Let $\vem \in L_{\omega}^{\rho_0} \xX(0,T_0)$ be the solution to \eqref{eq:trun_sub} with initial datum  $v_0 \in L_{\omega}^{\infty}H_{x}^{1}$. Then, we have the bound
	\begin{equation*}
	\|v_{m,\eps}\|_{L_{\omega}^{\rho_0}\xX^{1}(0,T_0)} \leq C \|v_0\|_{L_{\omega}^{\rho_0}H_{x}^{1}}. 
	\end{equation*}
\end{prop}

\begin{rmk}
	Note that even if one starts with $L_{\omega}^{\infty}H_x^1$ initial data, we can only have a bound for $\|v_{m,\eps}\|_{L_{\omega}^{\rho}H_x^1}$ for $\rho < +\infty$. This is already enough for us to compare two different nonlinearities, as it implies that the probability of having a large $H_x^1$ norm is small. 
\end{rmk}

\begin{rmk}
Proposition \ref{pr:uniform_stable} and Proposition \ref{pr:lwpinhigh} should be understood as the natural generalisations of \cite[Lemma~3.10, 3.12]{colliander2008global}. 
\end{rmk}

Thanks to the persistence of regularity, we can show the convergence of the solutions when starting from $L_{\omega}^{\infty} H_{x}^{1}$ initial data. This is the following proposition. 

\begin{prop} [Convergence with regular initial data]
	\label{pr:convergence_high}
	Let $v_{m,\eps}$ denotes the solutions to \eqref{eq:trun_sub} with initial data $v_0 \in L_{\omega}^{\infty}H_{x}^{1}$. Then, $\{v_{m,\eps}\}$ is Cauchy in $L_{\omega}^{\rho} \xX(0,T_0)$. 
\end{prop}

\begin{remark}
Note that the proof of Proposition~\ref{pr:converge_e} does not rely on the small initial data assumption. It is the uniform bound in both $\eps$ and $m$ in Proposition~\ref{pr:uniform_e_m} that requires the initial data to be small in $L_{\omega}^{\infty}L_{x}^{2}$. This uniformity in $m$ is essential for us to remove the truncation later in Section~\ref{sec:converge_m}. Hence, for consistency of the assumptions, we assume $\|u_0\|_{L_{\omega}^{\infty}L_{x}^{2}} \leq \delta_0$ throughout the article, including this section. This also enables us to use Proposition~\ref{pr:uniform_e_m} below rather than having another statement without assuming smallness of initial data but allowing dependence on $m$. 
\end{remark}

The three propositions above are all the ingredients we need. We will prove them in the next three subsections, and combine them together to prove Proposition~\ref{pr:converge_e} in the last subsection.

\subsection{Uniform stability -- proof of Proposition~\ref{pr:uniform_stable}}

The key ingredient to prove Proposition~\ref{pr:uniform_stable} is the following lemma.

\begin{lem} \label{le:stability_short}
	There exist $h>0$ and $C = C_m$ such that
	\begin{equation*}
	\|\uem - \vem\|_{L_{\omega}^{\rho_0} \xX(0,b)} \leq C \|u_{m,\eps} - v_{m,\eps}\|_{L_{\omega}^{\rho_0} \xX(0,a)}
	\end{equation*}
	whenever $[a,b] \subset [0,T_0]$ satisfies $b-a < h$. The bound is uniform over all $\eps \in (0,1)$. 
\end{lem}
\begin{proof}
	For every $a \leq \tau \leq t \leq b$, let
	\begin{equation*}
	\begin{split}
	\dD_{\tau}^{\gG}(t) = - i \int_{\tau}^{t} &\sS(t-s) \Big( \theta_{m}\big( \|u_{m,\eps}\|_{\xX_{2}(0,s)} \big) \nN^{\eps} \big(u_{m,\eps}(s)\big)\\
	&- \theta_{m}\big( \|v_{m,\eps}\|_{\xX_{2}(0,s)} \big) \nN^{\eps} \big(v_{m,\eps}(s)\big) \Big) {\rm d}s, 
	\end{split}
	\end{equation*}
	and
	\begin{equation*}
	\dD_{\tau}^{\qQ}(t) = - i \int_{\tau}^{t} \sS(t-s) \big( u_{m,\eps}(s) - v_{m,\eps}(s) \big) {\rm d} W_s. 
	\end{equation*}
	Note that as a process, $\dD_{\tau}^{\qQ}(\cdot)$ is defined pathwise even though the integral on the right hand side above is not. 
	
	For every $\omega \in \Omega$, we choose a dissection $\{\tau_k\}$ of the interval $[a,b]$ as follows. Let $\tau_0 = a$. Suppose $a = \tau_0 < \cdots < \tau_{k} < b$ is chosen, we choose $\tau_{k+1}$ by
	\begin{equation} \label{eq:stability_dissection}
	\begin{split}
	\tau_{k+1} \eqdef b \wedge \inf \Big\{ &r>\tau_k: \1_{\{\|u_{m,\eps}\|_{\xX_{2}(0,\tau_k)} < 2m\}} \|u_{m,\eps}\|_{\xX_{2}(\tau_k,r)}\\
	&+ \1_{\{\|v_{m,\eps}\|_{\xX_{2}(0,\tau_k)} < 2m\}} \|v_{m,\eps}\|_{\xX_{2}(\tau_k,r)} \geq \delta \Big\}, 
	\end{split}
	\end{equation}
	where $\delta>0$ is a small number to be specified later. In this way, we get a random dissection
	\begin{equation*}
	a = \tau_0 < \tau_1 < \cdots < \tau_K = b. 
	\end{equation*}
	Note that the total number of subintervals is always bounded by
	\begin{equation*}
	K \leq 2 \times (4m \delta^{-1})^{5}. 
	\end{equation*}
	For such a random dissection, we let $\iI_{k+1} = [\tau_{k}, \tau_{k+1}]$. For every $k = 0, \dots, K-1$, and every $t \in \iI_{k+1}$, we have
	\begin{equation*}
	\begin{split}
	u_{m,\eps}(t) - v_{m,\eps}(t) &= e^{i (t-\tau_k) \Delta} \big( u_{m,\eps}(\tau_k) - v_{m,\eps}(\tau_k) \big) + \dD_{\tau_k}^{\gG}(t) + \dD_{\tau_k}^{\qQ}(t)\\
	&- \frac{1}{2} \int_{\tau_k}^{t} \sS(t-s) \Big( F_{\Phi} \big( u_{m,\eps}(s) - v_{m,\eps}(s) \big) \Big) {\rm d}s. 
	\end{split}
	\end{equation*}
	We now control the $\xX(\iI_{k+1})$ norm of the four terms on the right hand side separately. For the term with the initial data, it follows immediately from Strichartz estimates that
	\begin{equation} \label{eq:stability_initial}
	\big\|e^{i (t-\tau_k) \Delta} \big( u_{m,\eps}(\tau_k) - v_{m,\eps}(\tau_k) \big) \big\|_{\xX(\iI_{k+1})} \leq C \|u_{m,\eps}-v_{m,\eps}\|_{\xX(0,\tau_k)}. 
	\end{equation}
	As for the correction term, by Proposition~\ref{pr:pre_correction}, we have
	\begin{equation} \label{eq:stability_correction}
	\Big\| \int_{\tau_k}^{t} \sS(t-s) \Big( F_{\Phi} \big( u_{m,\eps}(s) - v_{m,\eps}(s) \big) \Big) {\rm d}s \Big\|_{\xX(\iI_{k+1})} \leq C (b-a)^{\frac{4}{5}} \|u_{m,\eps} - v_{m,\eps}\|_{\xX(\iI_{k+1})}. 
	\end{equation}
	As for the stochastic term, we define
	\begin{equation*}
	\begin{split}
	M_{1}^{*}(t) &= \sup_{a \leq r_1 \leq r_2 \leq t} \Big\| \int_{r_1}^{r_2} \sS(t-s) \big( \uem(s) - \vem(s) \big) {\rm d} W_s \Big\|_{L_{x}^{2}};\\
	M_{2}^{*}(t) &= \sup_{a \leq r_1 \leq r_2 \leq t} \Big\| \int_{r_1}^{r_2} \sS(t-s) \big( \uem(s) - \vem(s) \big) {\rm d} W_s \Big\|_{L_{x}^{10}}. 
	\end{split}
	\end{equation*}
	We then have
	\begin{equation} \label{eq:stability_stochastic}
	\sup_{k} \|\dD_{\tau_k}^{\qQ}\|_{\xX(\iI_{k+1})} \leq \|M_{1}^{*}\|_{L_{t}^{\infty}(\iI)} + \|M_{2}^{*}\|_{L_{t}^{5}(\iI)} =: M_{a,b}^{*}, 
	\end{equation}
	where $\iI = [a,b]$. By Proposition~\ref{pr:pre_stochastic}, we have the bound
	\begin{equation} \label{eq:stability_stochastic_bound}
	\|M_{a,b}^{*}\|_{L_{\omega}^{\rho_0}} \leq C (b-a)^{\frac{3}{10}} \|u_{m,\eps} - v_{m,\eps}\|_{L_{\omega}^{\rho_0} \xX(a,b)}. 
	\end{equation}
	We finally turn to the nonlinearity $\dD_{\tau_k}^{\gG}$. For this, we consider situations according to whether $\|u_{m,\eps}\|_{\xX_{2}(0,\tau_k)}$ and $\|v_{m,\eps}\|_{\xX_{2}(0,\tau_k)}$ have reached $2m$ or not. 
	
	\begin{flushleft}
		\textit{Situation 1.}
	\end{flushleft}
	If both $\|u_{m,\eps}\|_{\xX_{2}(0,\tau_k)}$ and $\|v_{m,\eps}\|_{\xX_{2}(0,\tau_k)}$ are smaller than $2m$, then according to choice of $\tau_{k+1}$, we have
	\begin{equation} \label{eq:stability_increment}
	\|u_{m,\eps}\|_{\xX_{2}(\iI_{k+1})} + \|v_{m,\eps}\|_{\xX_{2}(\iI_{k+1})} \leq \delta. 
	\end{equation}
	As for the term inside the parenthesis after $\sS(t-s)$, we have the pointwise bound
	\begin{equation*}
	\begin{split}
	&\phantom{111}\Big| \theta_{m}\big(\|\uem\|_{\xX_{2}(0,s)}\big) \nN^{\eps}\big(u_{m,\eps}(s)\big) - \theta_{m}\big(\|\vem\|_{\xX_{2}(0,s)}\big) \nN^{\eps}\big(v_{m,\eps}(s)\big) \Big|\\
	&\leq \Big| \theta_{m}\big(\|\uem\|\big) \nN^{\eps}\big( \uem \big) - \theta_{m}\big(\|\vem\|\big) \nN^{\eps}\big(\uem\big) \Big|\\
	&+ \Big| \theta_{m}\big(\|\vem\|\big) \nN^{\eps}\big(u_{m,\eps}\big) - \theta_{m}\big(\|\vem\|\big) \nN^{\eps}\big(v_{m,\eps}\big) \Big|\\
	&\leq C \Big( \|u_{m,\eps} - v_{m,\eps}\|_{\xX_{2}(0,\tau_{k+1})} |u_{m,\eps}|^{5-\eps} + |\uem-\vem| \big( |\uem|^{4-\eps} + |\vem|^{4-\eps} \big)\Big), 
	\end{split}
	\end{equation*}
	where we have omitted the dependence of various quantities on $s$ after the first line, and also have made the relaxation
	\begin{equation*}
	\Big| \|u_{m,\eps}\|_{\xX_{2}(0,s)} - \|v_{m,\eps}\|_{\xX_{2}(0,s)} \Big| \leq \|u_{m,\eps} - v_{m,\eps}\|_{\xX_{2}(0,s)} \leq \|u_{m,\eps} - v_{m,\eps}\|_{\xX_{2}(0,\tau_{k+1})}
	\end{equation*}
	in the first term in the last inequality. Applying Strichartz estimates \eqref{eq:Strichartz_2} and making use of the small increment \eqref{eq:stability_increment}, we get
	\begin{equation*}
	\begin{split}
	\|\dD_{\tau_k}^{\gG}\|_{\xX(\iI_{k+1})} &\leq C \Big( \|u_{m,\eps}-v_{m,\eps}\|_{\xX(0,\tau_{k+1})} \|u_{m,\eps}\|_{\xX_{1}(\iI_{k+1})} \|u_{m,\eps}\|_{\xX_{2}(\iI_{k+1})}^{4-\eps}\\
	&\phantom{1111}+ \|u_{m,\eps}-v_{m,\eps}\|_{\xX_{1}(\iI_{k+1})} \big( \|u_{m,\eps}\|_{\xX_{2}(\iI_{k+1})} + \|v_{m,\eps}\|_{\xX_{2}(\iI_{k+1})}\big)\Big)\\
	&\leq C \Big( \|u_{m,\eps}-v_{m,\eps}\|_{\xX(0,\tau_{k})} + \delta^{4-\eps} \|\uem-\vem\|_{\xX(\iI_{k+1})} \Big), 
	\end{split}
	\end{equation*}
	where in the last line we have split $\|u_{m,\eps}-v_{m,\eps}\|_{\xX(0,\tau_{k+1})}$ into two disjoint regions $[0,\tau_k]$ and $\iI_{k+1}$, and merged the latter with the other term. 
	
	\begin{flushleft}
		\textit{Situation 2.}
	\end{flushleft}
    We turn to the case where $\|u_{m,\eps}\|_{\xX_{2}(0,\tau_k)} < 2m$ but $\|v_{m,\eps}\|_{\xX_{2}(0,\tau_k)} \geq 2m$. In this case, the choice of $\tau_{k+1}$ in \eqref{eq:stability_dissection} gives
    \begin{equation*}
    \|u_{m,\eps}\|_{\xX_{2}(\iI_{k+1})} \leq \delta. 
    \end{equation*}
    Since $\theta_{m}(\|v_{m,\eps}\|_{\xX_{2}(0,s)})$ vanishes for every $s>\tau_k$, we have the pointwise bound
    \begin{equation*}
    \begin{split}
    &\phantom{111}\Big| \theta_{m} \big(\|u_{m,\eps}\| \big) \nN^{\eps}(u_{m,\eps}) - \theta_{m} \big(\|v_{m,\eps}\| \big) \nN^{\eps}(v_{m,\eps}) \Big|\\
    &= \Big| \theta_{m} \big(\|u_{m,\eps}\| \big) \nN^{\eps}(u_{m,\eps}) - \theta_{m} \big(\|v_{m,\eps}\| \big) \nN^{\eps}(u_{m,\eps}) \Big|\\
    &\leq C \|u_{m,\eps} - v_{m,\eps}\|_{\xX_{2}(0,\tau_{k+1})} |u_{m,\eps}|^{5-\eps}, 
    \end{split}
    \end{equation*}
    where we have omitted the dependence of $\uem$ and $\vem$ on $s$ in the first two lines. Similar as before, applying Strichartz estimate \eqref{eq:Strichartz_2}, we get the bound
    \begin{equation*}
    \|\dD_{\tau_k}^{\gG}\|_{\xX(\iI_{k+1})} \leq C \delta_{0} \delta^{4-\eps} \Big( \|u_{m,\eps} - v_{m,\eps}\|_{\xX(0,\tau_k)} + \|u_{m,\eps} - v_{m,\eps}\|_{\xX(\iI_{k+1})} \Big). 
    \end{equation*}
    The right hand side above is symmetric in $\uem$ and $\vem$, so we have exactly the same bound in the case $\|\uem\|_{\xX_{2}(0,\tau_k)} \geq 2m$ but $\|\vem\|_{\xX_{2}(0,\tau_k)} < 2m$. 
    
    \begin{flushleft}
    	\textit{Situation 3.}
    \end{flushleft}
    If both $\|u_{m,\eps}\|_{\xX_{2}(0,\tau_k)}$ and $\|v_{m,\eps}\|_{\xX_{2}(0,\tau_k)}$ reaches $2m$, then both nonlinearities vanish, and hence $\|\dD_{\tau_k}^{\gG}\|_{\xX(\iI_{k+1})} = 0$. 
    
    \bigskip
    
    \bigskip
    
    Thus, in all of the situations, we always have the bound
    \begin{equation} \label{eq:stability_nonlinear}
    \|\dD_{\tau_k}^{\gG}\|_{\xX(\iI_{k+1})} \leq C \Big( \|u_{m,\eps}-v_{m,\eps}\|_{\xX(0,\tau_{k})} + \delta^{4-\eps} \|\uem-\vem\|_{\xX(\iI_{k+1})} \Big). 
    \end{equation}
    Combining \eqref{eq:stability_initial}, \eqref{eq:stability_correction}, \eqref{eq:stability_stochastic} and \eqref{eq:stability_nonlinear}, we get
    \begin{equation*}
    \begin{split}
    \|u_{m,\eps} - v_{m,\eps}\|_{\xX(\iI_{k+1})} \leq C &\Big( \big( (b-a)^{\frac{4}{5}} + \delta^{4-\eps} \big) \|u_{m,\eps} - v_{m,\eps}\|_{\xX(\iI_{k+1})}\\
    &+ \|u_{m,\eps}-v_{m,\eps}\|_{\xX(0,\tau_k)} \Big) + M_{a,b}^{*}. 
    \end{split}
    \end{equation*}
    If both $\delta$ and $b-a<h$ are small enough, we can absorb the term $\|u_{m,\eps}-v_{m,\eps}\|_{\xX(\iI_{k+1})}$ into the left hand side. This gives
    \begin{equation*}
    \|u_{m,\eps} - v_{m,\eps}\|_{\xX(\iI_{k+1})} \leq C \Big( \|u_{m,\eps} - v_{m,\eps}\|_{\xX(0,\tau_k)} + M_{a,b}^{*} \Big). 
    \end{equation*}
    Since there are at most $2 \times (4m \delta^{-1})^{5}$ subintervals in this dissection, iterating the above bound from for all the $\iI_k$'s and adding $\|u_{m,\eps} - v_{m,\eps}\|_{\xX(0,a)}$ to both sides, we obtain
    \begin{equation} \label{eq:stability_deterministic}
    \|u_{m,\eps} - v_{m,\eps}\|_{\xX(0,b)} \leq C \Big( \|u_{m,\eps} - v_{m,\eps}\|_{\xX(0,a)} + M_{a,b}^{*} \Big). 
    \end{equation}
    So far, all the arguments are deterministic, and the bound \eqref{eq:stability_deterministic} holds almost surely. Moreover, the proportionality constant $C$ is deterministic, and depends on $m$ only. 
    
    We now take $L_{\omega}^{\rho_0}$ norm on both sides of \eqref{eq:stability_deterministic}. By \eqref{eq:stability_stochastic_bound}, if $h$ is small enough, we can again absorb the term $\|u_{m,\eps}-v_{m,\eps}\|_{L_{\omega}^{\rho_0}\xX(a,b)}$ arising from $\|M_{a,b}^{*}\|_{L_{\omega}^{\rho_0}}$ into the left hand side. Hence, we finally have the bound
    \begin{equation*}\|u_{m,\eps} - v_{m,\eps}\|_{L_{\omega}^{\rho_0}\xX(0,b)} \leq C \|u_{m,\eps} - v_{m,\eps}\|_{L_{\omega}^{\rho_0}\xX(0,a)}. 
    \end{equation*}
    This completes the proof. 
\end{proof}

We are now ready to prove Proposition~\ref{pr:uniform_stable}. 

\begin{proof} [Proof of Proposition~\ref{pr:uniform_stable}]
	The key is to note that the smallness of $h$ in order for Lemma~\ref{le:stability_short} to be true depends only on the universal constants from Strichartz estimates, and hence is also universal itself. We can then iterate Lemma~\ref{le:stability_short} with small intervals of length $h$ up to time $T_0$. This completes the proof. 
\end{proof}

\subsection{Persistence of  regularity -- proof of Proposition~\ref{pr:lwpinhigh}}

Let $v_0 \in L_{\omega}^{\infty}H_{x}^{1}$ with $\|v_0\|_{L_{\omega}^{\infty}L_{x}^{2}} \leq \delta_0$. Recall that the $\|\cdot\|_{\xX^1(\iI)}$ norm is given by
\begin{equation*}
\|v\|_{\xX^1(\iI)} := \|v\|_{\xX(\iI)} + \|\d_x v\|_{\xX(\iI)}. 
\end{equation*}
The local existence of the $H^1$ solution is standard. More precisely, there exists $R_0 > 0$ depending on $\|v_0\|_{L_{\omega}^{\rho_0}H_{x}^{1}}$ only such that for every $\eps>0$, there is a unique $v_{m,\eps} \in L_{\omega}^{\rho_0} \xX^{1}(0,R_0)$ that solves \eqref{eq:trun_sub} on $[0,R_0]$ with initial data $v_0$. We want to show that the solution actually exists on $[0,T_0]$ (if $R_0 < T_0$), and satisfies the bound
\begin{equation*}
\|v_{m,\eps}\|_{L_{\omega}^{\rho_0} \xX^{1}(0,T_0)} \lesssim_{m} \|v_0\|_{L_{\omega}^{\rho_0} H_x^1}
\end{equation*}
uniformly in $\eps$. The key ingredient is the following lemma, which shows the persistence of regularity of $H^1$ solutions. 

\begin{lem} \label{le:persistence_short}
	There exist $h>0$ and $C = C_m$ such that if $R \leq T_0$ and $v_{m,\eps}$ solves \eqref{eq:trun_sub} in $L_{\omega}^{\rho_0} \xX^{1}(0,R)$ with initial data $v_0$, then we have the bound
	\begin{equation} \label{eq:persistence_short}
	\|v_{m,\eps}\|_{L_{\omega}^{\rho_0}\xX^{1}(0,b)} \leq C \|v_{m,\eps}\|_{L_{\omega}^{\rho_0} \xX^{1}(0,a)}
	\end{equation}
	whenever $[a,b] \subset [0,R]$ satisfies $b-a<h$. As a consequence, we have
	\begin{equation} \label{eq:persistence_medium}
	\|v_{m,\eps}\|_{L_{\omega}^{\rho_0} \xX^{1}(0,R)} \leq C^{1+\frac{T_0}{h}} \|v_0\|_{L_{\omega}^{\rho_0}H_{x}^{1}}. 
	\end{equation}
	Here, the constant $C$ depends on $m$ but is uniform over all $\eps \in (0,1)$ and all $R \leq T_0$. 
\end{lem}
\begin{proof}
	The proof is essentially the same as that for Lemma~\ref{le:stability_short}. Let
	\begin{equation*}
	\begin{split}
	M_{1}^{**}(t) &= \sup_{a \leq r_1 \leq r_2 \leq t} \Big\| \int_{r_1}^{r_2} \sS(t-s) v_{m,\eps}(s) {\rm d} W_s \Big\|_{H_{x}^{1}}, \\
	M_{2}^{**}(t) &= \sup_{a \leq r_1 \leq r_2 \leq t} \Big\| \int_{r_1}^{r_2} \sS(t-s) v_{m,\eps}(s) {\rm d} W_s \Big\|_{W_{x}^{1,10}}, 
	\end{split}
	\end{equation*}
	where $\|f\|_{W_{x}^{1,10}} = \|f\|_{L_{x}^{10}} + \|\d_x f\|_{L_{x}^{10}}$. Let
	\begin{equation*}
	M_{a,b}^{**} := \|M_{1}^{**}\|_{L_{t}^{\infty}(a,b)} + \|M_{2}^{**}\|_{L_{t}^{5}(a,b)}. 
	\end{equation*}
	Since the differentiation in $x$ variable commutes with the operator $\sS(t-s)$, in addition to the $L_x^2$ and $L_{x}^{10}$ norms of the above quantities, we also need to control
	\begin{equation*}
	\int_{r_1}^{r_2} \sS(t-s) \d_x v_{m,\eps}(s) {\rm d} W_s \quad \text{and} \quad \int_{r_1}^{r_2} \sS(t-s) v_{m,\eps}(s) {\rm d} \d_x W_s.  
	\end{equation*}
	The only difference to Proposition~\ref{pr:pre_stochastic} is that in the latter situation, the noise $W$ is replaced by $\d_x W$. This amounts to replace $\|\Phi\|$ in that proposition by $\|\d_x \Phi\|$ with the same norm, which is also finite by Assumption~\ref{as:Phi} and Lemma~\ref{le:factorisation}. Hence, it follows from Proposition~\ref{pr:pre_stochastic} that
	\begin{equation} \label{eq:persistence_stochastic}
	\|M_{a,b}^{**}\|_{L_{\omega}^{\rho_0}} \leq C (b-a)^{\frac{3}{10}} \|v^\eps\|_{L_{\omega}^{\rho_0} \xX^{1}(a,b)}. 
	\end{equation}
	Now, similar to the argument in Lemma~\ref{le:stability_short}, we choose a random dissection $\{\tau_k\}$ of the interval $[a,b]$ as follows. Let $\tau_0=a$, and define
	\begin{equation*}
	\tau_{k+1} = b \wedge \inf \big\{r > \tau_k: \1_{\{\|v_{m,\eps}\|_{\xX_2(0,\tau_k)} \leq 2m\}} \|v_{m,\eps}\|_{\xX_2(\tau_k,r)} \geq \delta \big\}. 
	\end{equation*}
	We then have a dissection
	\begin{equation*}
	a = \tau_0 < \tau_1 < \cdots < \tau_K = b, 
	\end{equation*}
	and the total number of points $K$ is at most $(4m \delta^{-1})^{5}$. Let $\iI_{k+1} = [\tau_k,\tau_{k+1}]$. Again, since $\d_x$ commutes with $\sS(t-s)$ and
	\begin{equation*}
	|\d_x \nN^{\eps}(v_{m,\eps})| \leq C |\d_x v_{m,\eps}| \cdot |v_{m,\eps}|^{4-\eps}, 
	\end{equation*}
	the $\xX^{1}(\iI_{k+1})$ norm of the nonlinearity can be controlled by $\delta^{4-\eps} \|v_{m,\eps}\|_{\xX_{1}(\iI_{k+1})}$. We then have the pathwise bound
	\begin{equation*}
	\begin{split}
	\|v_{m,\eps}\|_{\xX^{1}(\iI_{k+1})} \leq C &\Big( \|v_{m,\eps}\|_{\xX^{1}(0,\tau_k)} + \delta^{4-\eps} \|v_{m,\eps}\|_{\xX^{1}(\iI_{k+1})}\\
	&+ (b-a)^{\frac{4}{5}} \|v_{m,\eps}\|_{\xX^{1}(\iI_{k+1})} \Big) + M_{a,b}^{**}.
	\end{split} 
	\end{equation*}
	If both $h$ and $\delta$ are sufficiently small, we can absorb the term $\|v_{m,\eps}\|_{\xX^{1}(\iI_{k+1})}$ into the left hand side. Since the number of the $\tau_k$'s is at most $(4 m \delta)^{-1}$, we can iterate the above bound over the intervals $\{\iI_{k}\}$ and then adding $\|v_{m,\eps}\|_{\xX^{1}(a,b)}$ to both sides so that we obtain the bound
	\begin{equation*}
	\|v_{m,\eps}\|_{\xX^{1}(0,b)} \leq C \Big( \|v_{m,\eps}\|_{\xX^{1}(0,a)} + M_{a,b}^{**} \Big). 
	\end{equation*}
	Now, we take $L_{\omega}^{\rho_0}$-norm on both sides. By \eqref{eq:persistence_stochastic}, if $h$ is sufficiently small, we can absorb the term $\|M_{a,b}^{**}\|_{L_{\omega}^{\rho_0}}$ to the left hand side to obtain \eqref{eq:persistence_short}. 
	
	As for the second part of the claim, we iterate the bound \eqref{eq:persistence_short} in the interval $[0,R]$ for at most $1 + \frac{R}{h}$ steps. The bound \eqref{eq:persistence_medium} then follows. 
\end{proof}

\begin{proof} [Proof of Proposition~\ref{pr:lwpinhigh}]
	From local existence of $H^1$ solutions, we know there exists a unique $v_{m,\eps}$ that solves \eqref{eq:trun_sub} in $L_{\omega}^{\rho_0} \xX^{1}(0,R_0)$ with initial data $v_0$, where $R_0>0$ depends on $\|v_0\|_{L_{\omega}^{\rho_0}H_{x}^{1}}$ only. If $R_0 \geq T_0$, then Lemma~\ref{le:persistence_short} immediately implies the desired bound. If $R_0 < T_0$, then by the same lemma, we have
	\begin{equation*}
	\|v_{m,\eps}(R_0/2)\|_{L_{\omega}^{\rho_0}H_{x}^{1}} \leq C^{*} \|v_0\|_{L_{\omega}^{\rho_0} H_{x}^{1}}, 
	\end{equation*}
	where $C^{*}$ depends on $m$ only (we omit its dependence on $T_0$ since it is a fixed number). Starting from $\frac{R_0}{2}$, we can extend the solution to a period of time which depends on $C^{*} \|v_0\|_{L_{\omega}^{\rho_0}H_{x}^{1}}$ only. Lemma~\ref{le:persistence_short} ensures that we can repeat this procedure with the same extension time at every step. The procedure necessarily ends in finitely many steps when we reach $T_0$. This shows that the solution $v_{m,\eps}$ can be extended to $T_0$ and it satisfies the desired bound. 
\end{proof}

\subsection{Convergence with regular initial data -- proof of Proposition~\ref{pr:convergence_high}}

Let $v_{m,\eps_1}$ and $v_{m,\eps_2}$ denote the solutions to \eqref{eq:trun_sub} with common initial data $v_0 \in L_{\omega}^{\infty} H_{x}^{1}$ and nonlinearities $\nN^{\eps_1}$ and $\nN^{\eps_2}$ respectively. 

\begin{lem} \label{le:Cauchy_high_local}
	There exists $h>0$ such that for every $\eta>0$ and every $v_0 \in L_{\omega}^{\infty}H_{x}^{1}$ with $\|v_0\|_{L_{\omega}^{\infty}L_{x}^{2}} \leq \delta_0$, there exist $\eps^{*}>0$ and $\kappa>0$ such that for every $\eps_1, \eps_2 < \eps^{*}$ and every $[a,b] \subset [0,T_0]$ with $b-a < h$, if
	\begin{equation*}
	\|v_{m,\eps_1} - v_{m,\eps_2}\|_{L_{\omega}^{\rho}(0,a)} < \kappa, 
	\end{equation*}
	then we have the bound
	\begin{equation*}
	\|v_{m,\eps_1} - v_{m,\eps_2}\|_{L_{\omega}^{\rho}(0,b)} < \eta. 
	\end{equation*}
\end{lem}
\begin{proof}
	The proof is essentially the same as that for Lemma~\ref{le:stability_short}, but with one more ingredient. For every $\omega \in \Omega$, let
	\begin{equation*}
	a = \tau_0 < \tau_1 < \cdots < \tau_K = b
	\end{equation*}
	be a dissection of the interval $[a,b]$ chosen according to \eqref{eq:stability_dissection}, except that one replaces $u_{m,\eps}$ and $v_{m,\eps}$ by $v_{m,\eps_1}$ and $v_{m,\eps_2}$ in this case. With the same replacement, the bounds \eqref{eq:stability_initial}, \eqref{eq:stability_correction} and \eqref{eq:stability_stochastic} also carry straightforwardly, so we have
	\begin{equation*}
	\begin{split}
	\|v_{m,\eps_1} - v_{m,\eps_2}\|_{\xX(\iI_{k+1})} &\leq C \Big( \|v_{m,\eps_1} - v_{m,\eps_2}\|_{\xX(0,\tau_k)} + (b-a)^{\frac{4}{5}} \|v_{m,\eps_1} - v_{m,\eps_2}\|_{\xX(\iI_{k+1})} \Big)\\
	&\phantom{11}+ M_{a,b}^{*} + \|\dD_{\tau_k}^{\gG}\|_{\xX(\iI_{k+1})}, 
	\end{split}
	\end{equation*}
	where $\iI_{k+1} = [\tau_k, \tau_{k+1}]$, 
	\begin{equation*}
	\begin{split}
	\dD_{\tau_k}^{\gG}(t) &= - i \int_{\tau_k}^{t} \sS(t-s) \Big( \theta_{m}\big(\|v_{m,\eps_1}\|_{\xX_{2}(0,s)}\big) \nN^{\eps_1} \big(v_{m,\eps_1}(s)\big)\\
	&\phantom{11}- \theta_{m}\big(\|v_{m,\eps_2}\|_{\xX_{2}(0,s)}\big) \nN^{\eps_2} \big(v_{m,\eps_2}(s)\big) \Big) {\rm d}s, 
	\end{split}
	\end{equation*}
	and $M_{a,b}^{*} \geq 0$ satisfies the moment bound
	\begin{equation} \label{eq:cauchy_stochastic}
	\|M_{a,b}^{*}\|_{L_{\omega}^{\rho_0}} \leq C (b-a)^{\frac{3}{10}} \|v_{m,\eps_1} - v_{m,\eps_2}\|_{L_{\omega}^{\rho_0} \xX(a,b)}. 
	\end{equation}
	Again, if $h$ is sufficiently small, we can absorb the term $\|v_{m,\eps_1} - v_{m,\eps_2}\|_{\xX(\iI_{k+1})}$ into the left hand side so that
	\begin{equation} \label{eq:cauchy_deter_1}
	\|v_{m,\eps_1} - v_{m,\eps_2}\|_{\xX(\iI_{k+1})} \leq C \Big( \|v_{m,\eps_1} - v_{m,\eps_2}\|_{\xX(0,\tau_k)} + M_{a,b}^{*} + \|\dD_{\tau_k}^{\gG}\|_{\xX(\iI_{k+1})} \Big). 
	\end{equation}
	So far it has been exactly the same as in Lemma~\ref{le:stability_short}, and the bound \eqref{eq:cauchy_deter_1} holds pathwise with a deterministic constant $C$. Now the difference comes in the bound for $\dD_{\tau_k}^{\gG}$, as the two nonlinearities have different powers, so we need information of the $L_{x}^{\infty}$ norms of $v^{m,\eps_j}$ to control their difference. 
	
	We now start to control $\dD_{\tau_k}^{\gG}$. For every $\Lambda > 1$, let
	\begin{equation*}
	\Omega_{\Lambda} = \big\{ \omega \in \Omega: \|v_{m,\eps_1}\|_{L_{t}^{\infty}H_{x}^{1}(0,T_0)} \leq \Lambda,\; \|v_{m,\eps_2}\|_{L_{t}^{\infty}H_{x}^{1}(0,T_0)} \leq \Lambda \big\}. 
	\end{equation*}
	By Proposition~\ref{pr:lwpinhigh}, there exists $C > 0$ such that
	\begin{equation} \label{eq:cauchy_prob}
	\Pr (\Omega_{\Lambda}^{c}) \leq \frac{C}{\Lambda^{\rho_0}}
	\end{equation}
	for every $\Lambda > 1$. On $\Omega_{\Lambda}$, we still bound the nonlinearity pathwise. For every $\omega \in \Omega_{\Lambda}$, we write
	\begin{equation*}
	\begin{split}
	&\phantom{111}\theta_{m}\big(\|v_{m,\eps_1}\|\big) \nN^{\eps_1}\big(v_{m,\eps_1}\big) - \theta_{m}\big(\|v_{m,\eps_2}\|\big) \nN^{\eps_2}\big(v_{m,\eps_2}\big)\\
	&= \Big( \theta_{m}\big(\|v_{m,\eps_1}\|\big) \nN^{\eps_1}\big(v_{m,\eps_1}\big) - \theta_{m}\big(\|v_{m,\eps_2}\|\big) \nN^{\eps_1}\big(v_{m,\eps_2}\big) \Big)\\
	&+ \Big( \theta_{m}\big(\|v_{m,\eps_2}\|\big) \nN^{\eps_1}\big(v_{m,\eps_2}\big) - \theta_{m}\big(\|v_{m,\eps_2}\|\big) \nN^{\eps_2}\big(v_{m,\eps_2}\big) \Big),
	\end{split} 
	\end{equation*}
	and we write $\dD_{\tau_k}^{\gG} = \dD_{\tau_k,1}^{\gG} + \dD_{\tau_k,2}^{\gG}$, corresponding to the two terms in the above decomposition of the nonlinearity respectively. For $\dD_{\tau_k,1}^{\gG}$, since the two nonlinearities are the same (both with $\eps_1$), we have exactly the same bound as in \eqref{eq:stability_nonlinear}. Hence, if $\delta$ is small, we have
	\begin{equation} \label{eq:cauchy_N_1}
	\|\dD_{\tau_k,1}^{\gG}\|_{\xX(\iI_{k+1})} \leq C \|v_{m,\eps_1} - v_{m,\eps_2}\|_{\xX(0,\tau_k)}. 
	\end{equation}
	For $\dD_{\tau_k,2}^{\gG}$, we note that the two nonlinearities have different powers but the same input $v_{m,\eps_2}$. Since $L_{t}^{1} L_{x}^{2}$ is also the dual of a Strichartz pair, applying \eqref{eq:Strichartz_2}, we have
	\begin{equation*}
	\begin{split}
	\|\dD_{\tau_k,2}^{\gG}\|_{\xX(\iI_{k+1})} &\leq C \big\| |v_{m,\eps_2}| \cdot \big( |v_{m,\eps_2}|^{4-\eps_1} - |v_{m,\eps_2}|^{4-\eps_2} \big) \big\|_{L_{t}^{1}L_{x}^{2}(\iI_{k+1})}\\
	&\leq C \|v_{m,\eps_2}\|_{\xX_{1}(\iI_{k+1})} \big\| \, |v_{m,\eps_2}|^{4-\eps_1} - |v_{m,\eps_2}|^{4-\eps_2} \big\|_{L_{t}^{\infty} L_{x}^{\infty}(\iI_{k+1})}. 
	\end{split}
	\end{equation*}
	Since we are in dimension one, we have
	\begin{equation*}
	\|v_{m,\eps_j}\|_{L_{x}^{\infty}}^{2} \leq 2 \|v_{m,\eps_j}\|_{L_{x}^{2}} \|\d_x v_{m,\eps_j}\|_{L_{x}^{2}} \leq 2 \Lambda^{2}, 
	\end{equation*}
	where the first inequality follows from Newton-Leibniz for the function $(v_{m,\eps_j})^{2}$, and in the second inequality we have used the assumption that $\omega \in \Omega_{\Lambda}$. Plugging this back to the bound for $\dD_{\tau_k,2}^{\gG}$ above, we get
	\begin{equation} \label{eq:cauchy_N_2}
	\|\dD_{\tau_k,2}^{\gG}\|_{\xX(\iI_{k+1})} \leq C \Lambda^{4} |\eps_1 - \eps_2| \leq C \Lambda^{4} \eps^{*}, 
	\end{equation}
	where we have replaced $\|v_{m,\eps_2}\|_{\xX_{1}(\iI_{k+1})}$ by a constant due to the conservation law. Plugging \eqref{eq:cauchy_N_1} and \eqref{eq:cauchy_N_2} back into \eqref{eq:cauchy_deter_1}, we get
	\begin{equation*}
	\|v_{m,\eps_1} - v_{m,\eps_2}\|_{\xX(\iI_{k+1})} \leq C \Big( \|v_{m,\eps_1} - v_{m,\eps_2}\|_{\xX(0,\tau_k)} + \Lambda^{4} \eps^{*} + M_{a,b}^{*} \Big). 
	\end{equation*}
	Again, iterating this bound over the intervals $\iI_1, \dots, \iI_K$ and adding $\|v_{m,\eps_1} - v_{m,\eps_2}\|_{\xX(0,a)}$ to both sides, we get
	\begin{equation*}
	\|v_{m,\eps_1} - v_{m,\eps_2}\|_{\xX(0,b)} \leq C_{m} \Big( \|v_{m,\eps_1} - v_{m,\eps_2}\|_{\xX(0,a)} + \Lambda^{4} \eps^{*} + M_{a,b}^{*} \Big) \quad \text{on} \phantom{1} \Omega_{\Lambda}. 
	\end{equation*}
	Taking $L_{\omega}^{\rho_0}$-norm restricted to $\Omega_{\Lambda}$ and using the bound \eqref{eq:cauchy_stochastic}, we get
	\begin{equation} \label{eq:cauchy_average_1}
	\begin{split}
	&\phantom{111}\|(v_{m,\eps_1} - v_{m,\eps_2}) \1_{\Omega_{\Lambda}}\|_{L_{\omega}^{\rho_0} \xX(0,b)}\\
	&\leq C \Big( \|v_{m,\eps_1}-v_{m,\eps_2}\|_{L_{\omega}^{\rho_0} \xX(0,a)} + \Lambda^{4} \eps^{*} + h^{\frac{3}{10}} \|v_{m,\eps_1} - v_{m,\eps_2}\|_{L_{\omega}^{\rho_0} \xX(a,b)} \Big), 
	\end{split}
	\end{equation}
	where the constant $C$ is deterministic and depends on $m$ only. On $\Omega_{\Lambda}^{c}$, we have
	\begin{equation} \label{eq:cauchy_average_2}
	\|(v_{m,\eps_1}-v_{m,\eps_2}) \1_{\Omega_{\Lambda}^{c}}\|_{L_{\omega}^{\rho_0} \xX(0,b)} \leq \|v_{m,\eps_1}-v_{m,\eps_2}\|_{L_{\omega}^{2\rho_0}\xX(0,b)} \big( \Pr(\Omega_{\Lambda}^{c}) \big)^{\frac{1}{2\rho}} \leq \frac{C}{\sqrt{\Lambda}}, 
	\end{equation}
	where we have used Proposition~\ref{pr:uniform_e_m} with $\rho = 2 \rho_0$ and \eqref{eq:cauchy_prob} to control the two terms after using H\"{o}lder. Adding \eqref{eq:cauchy_average_1} and \eqref{eq:cauchy_average_2} together, if $h$ is sufficiently small, we can again absorb the term $\|v_{m,\eps_1} - v_{m,\eps_2}\|_{L_{\omega}^{\rho_0} \xX(a,b)}$ to the left hand side. This gives us the bound
	\begin{equation} \label{eq:cauchy_local}
	\|v_{m,\eps_1} - v_{m,\eps_2}\|_{L_{\omega}^{\rho_0}\xX(0,b)} \leq C_{m} \Big( \|v_{m,\eps_1} - v_{m,\eps_2}\|_{L_{\omega}^{\rho_0}\xX(0,a)} + \Lambda^{4} \eps^{*} + \Lambda^{-\frac{1}{2}} \Big). 
	\end{equation}
	Now, for every $\eta>0$, we choose $\Lambda$ large enough so that $C_{m} \Lambda^{-\frac{1}{2}} < \frac{\eta}{3}$, then $\eps^{*}$ small so that $C_{m} \Lambda^{4} \eps^{*} < \frac{\eta}{3}$, and finally $\kappa$ small so that the first term on the right hand side is also bounded by $\frac{\eta}{3}$. This completes the proof of the lemma. 
\end{proof}

\begin{proof} [Proof of Proposition~\ref{pr:convergence_high}]
	It suffices to show that for every $\eta>0$, there exists $\eps^{*}>0$ such that
	\begin{equation} \label{eq:high_aim}
	\|v_{m,\eps_1} - v_{m,\eps_2}\|_{L_{\omega}^{\rho_0}\xX(0,T_0)} < \eta
	\end{equation}
	whenever $\eps_1, \eps_2 < \eps^{*}$. The proof is a backward bootstrap argument using Lemma \ref{le:Cauchy_high_local}. For every $\eta>0$, there exists $\eps^{(1)}$ and $\kappa^{(1)}$ such that if $\eps_1, \eps_2 < \eps^{(1)}$, then $\|v_{m,\eps_1} - v_{m,\eps_2}\|_{L_{\omega}^{\rho_0} \xX(0,T-h)} < \kappa^{(1)}$ implies $\|v_{m,\eps_1} - v_{m,\eps_2}\|_{L_{\omega}^{\rho_0}\xX(0,T_0)} < \eta$. Suppose there exists $\kappa^{(n)}>0$ and $\eps^{(n)}$ such that \eqref{eq:high_aim} holds whenever $\eps_1, \eps_{1}" < \eps^{(n)}$ and 
	\begin{equation} \label{eq:high_n}
	\|v_{m,\eps_1}-v_{m,\eps_2}\|_{L_{\omega}^{\rho_0}\xX(0,T-nh)} < \kappa^{(n)}. 
	\end{equation}
	By Lemma~\ref{le:Cauchy_high_local}, we can find $\eps^{(n+1)}$ and $\kappa^{(n+1)}$ such that $\eps_1, \eps_2 < \eps^{(n+1)}$ and $\|v_{m,\eps_1}-v_{m,\eps_2}\|_{L_{\omega}^{\rho_0}\xX(0,T-(n+1)h)} < \kappa^{(n+1)}$ implies \eqref{eq:high_n}, which then further implies \eqref{eq:high_aim}. This procedure necessarily ends in finitely many steps depending on $h$ and $T_0$ only. We then take
	\begin{equation*}
	\eps^{*} = \min_{j} \{\eps^{(j)}\}. 
	\end{equation*}
	The proof is complete by noting that $v_{m,\eps_1}$ and $v_{m,\eps_2}$ have the same initial data so that they start with zero difference. 
\end{proof}

\subsection{Proof of Proposition~\ref{pr:converge_e}}

We are now ready to prove the convergence in $\eps$. Let $u_{m,\eps_1}$ and $u_{m,\eps_2}$ denote the solutions to \eqref{eq:trun_sub} with common initial data $u_0$ and nonlinearities $\nN^{\eps_1}$ and $\nN^{\eps_2}$ respectively. We first need the following lemma to perturb the initial data to $H_{x}^{1}$. 

\begin{lem} \label{le:initial_smoothen}
	For every $u_0 \in L_{\omega}^{\infty}L_{x}^{2}$ and every $\kappa > 0$, there exists $v_0 \in L_{\omega}^{\infty}H_{x}^{1}$ such that $\|v_0\|_{L_x^2} \leq \|u_0\|_{L_x^2}$ almost surely, and
	\begin{equation*}
	\|u_0 - v_0\|_{L_{\omega}^{\rho_0}L_{x}^{2}} < \kappa_0. 
	\end{equation*}
\end{lem}
\begin{proof}
	We fix $u_0 \in L_{\omega}^{\infty}L_{x}^{2}$ and $\kappa > 0$ arbitrary. Let $\varphi$ be a mollifier on $\RR$, and $\varphi_{\delta} = \delta^{-1} \varphi(\cdot / \delta)$. For every $\omega \in \Omega$, let $u_{0,\delta}(\omega) = u_0(\omega) * \varphi_{\delta}$. By Young's inequality, we have
	\begin{equation} \label{eq:initial_convolution}
	\|u_{0,\delta}(\omega)\|_{L_{x}^{2}} \leq \|u_{0}(\omega)\|_{L_{x}^{2}}, \qquad \|u_{0,\delta}(\omega)\|_{\dot{H}_{x}^{1}} \leq \delta^{-1} \|\varphi'\|_{L_{x}^{1}} \|u_{0}(\omega)\|_{L_{x}^{2}}
	\end{equation}
	for almost every $\omega$ and every $\delta>0$. In addition, 
	\begin{equation*}
	\|u_{0,\delta}(\omega) - u_{0}(\omega)\|_{L_{x}^{2}} \rightarrow 0
	\end{equation*}
	as $\delta \rightarrow 0$ for almost every $\omega$. By Egorov's theorem, there exists $\Omega' \subset \Omega$ with $\Pr(\Omega') < \kappa / 4 \|u_0\|_{L_{\omega}^{\infty}L_{x}^{2}}$ such that on $\Omega \setminus \Omega'$, we have
	\begin{equation*}
	\sup_{\omega \in \Omega \setminus \Omega'} \|u_{0,\delta}(\omega) - u_{0}(\omega)\|_{L_{x}^{2}} \rightarrow 0
	\end{equation*}
	as $\delta \rightarrow 0$. Thus, we can choose $\delta = \delta^{*}$ small enough so that
	\begin{equation*}
	\sup_{\omega \in \Omega \setminus \Omega'} \|u_{0,\delta^{*}}(\omega) - u_{0}(\omega)\|_{L_{x}^{2}} < \frac{\kappa}{2}. 
	\end{equation*}
	Let $v_{0}(\omega) = u_{0,\delta^{*}}(\omega)$. It then follows immediately from \eqref{eq:initial_convolution} that $v_0 \in L_{\omega}^{\infty} H_{x}^{1}$ and $\|v_0\|_{L_{x}^{2}} \leq \|u_0\|_{L_{x}^{2}}$ almost surely. As for $\|u_0-v_0\|_{L_{\omega}^{\rho_0}L_{x}^{2}}$, we have
	\begin{equation*}
	\begin{split}
	\|u_0 - v_0\|_{L_{\omega}^{\rho_0} L_{x}^{2}} &\leq \|(u_0-v_0) \1_{\Omega \setminus \Omega'}\|_{L_{\omega}^{\rho_0}L_{x}^{2}} + \|(u_0-v_0) \1_{\Omega'}\|_{L_{\omega}^{\rho_0}L_{x}^{2}}\\
	&\leq \frac{\kappa}{2} + \big( \|u_0\|_{L_{\omega}^{\infty}L_{x}^{2}} + \|v_0\|_{L_{\omega}^{\infty}L_{x}^{2}} \big) \Pr(\Omega')\\
	&\leq \kappa. 
	\end{split}
	\end{equation*}
	This completes the proof. 
\end{proof}

\begin{proof} [Proof of Proposition~\ref{pr:converge_e}]
	We first show that, for every fixed $m$, the sequence $\{u_{m,\eps}\}_{\eps>0}$ is Cauchy in $L_{\omega}^{\rho_0} \xX(0,T_0)$. Fix an arbitrary $\eta>0$. For every $\kappa>0$, by Lemma~\ref{le:initial_smoothen}, we can choose $v_0 \in L_{\omega}^{\infty}H_{x}^{1}$ such that
	\begin{equation*}
	\|v_0\|_{L_{\omega}^{\infty}L_{x}^{2}} \leq \delta_0, \qquad \|u_0 - v_0\|_{L_{\omega}^{\rho_0}L_{x}^{2}} < \kappa. 
	\end{equation*}
	For every $\eps_1$ and $\eps_2$, we have
	\begin{equation*}
	\begin{split}
	\|u_{m,\eps_1} - u_{m,\eps_2}\|_{L_{\omega}^{\rho_0}\xX(0,T_0)} \leq &\|u_{m,\eps_1} - v_{m,\eps_1}\|_{L_{\omega}^{\rho_0}\xX(0,T_0)} + \|u_{m,\eps_2} - v_{m,\eps_2}\|_{L_{\omega}^{\rho_0}\xX(0,T_0)}\\
	&+ \|v_{m,\eps_1} - v_{m,\eps_2}\|_{L_{\omega}^{\rho_0} \xX(0,T_0)}. 
	\end{split}
	\end{equation*}
	By Proposition~\ref{pr:uniform_stable}, we can let $\kappa$ be sufficiently small so that the first two terms on the right hand side above are both smaller than $\frac{\eta}{3}$, uniformly in $\eps_1, \eps_2 \in (0,1)$. As for the third term, by Proposition~\ref{pr:convergence_high}, we can choose $\eps^{*}$ sufficiently small so that this term is also smaller than $\frac{\eta}{3}$ as long as $\eps_1, \eps_2 < \eps^{*}$. This shows that $u_{m,\eps}$ converges in $L_{\omega}^{\rho} \xX(0,T)$ to a limit, which we denote by $u_{m}$. 
	
	It then remains to show that the limit $u_{m}$ satisfies the equation \eqref{eq:trun_critical}. To see this, it suffices to show that each term on the right hand side of \eqref{eq:duhamel_trun_sub} converges in $L_{\omega}^{\rho_0}\xX(0,T_0)$ to the corresponding term with $u_{m,\eps}$ replaced by $u_{m}$. The term with the initial data are identical for all $\eps \geq 0$. The convergence of
	\begin{equation*}
	\int_{0}^{t} \sS(t-s) u_{m,\eps}(s) {\rm d} W_s \quad \text{and} \quad \int_{0}^{t} \sS(t-s) \big( F_{\Phi} u_{m,\eps}(s) \big) {\rm d}s
	\end{equation*}
	follows immediately from the bounds in Propositions~\ref{pr:pre_stochastic} and~\ref{pr:pre_correction} and that $\|u_{m,\eps}-u_{m}\|_{L_{\omega}^{\rho_0}\xX(0,T_0)} \rightarrow 0$. As for the term with the nonlinearity, one needs to compare $\nN^{\eps}(u_{m,\eps})$ and $\nN(u_{m})$. Hence, we proceed as above to perturb the initial data to $v_0 \in H_x^1$. We then write
	\begin{equation*}
	\begin{split}
	&\phantom{111}\theta_{m}\big(\|u_{m,\eps}\|\big) \nN^{\eps}(u_{m,\eps}) - \theta_{m} \big(\|u_{m}\|\big) \nN(u)\\
	&= \theta_{m}\big(\|u_{m,\eps}\|\big) \nN^{\eps}(u_{m,\eps}) - \theta_{m} \big(\|v_{m,\eps}\|\big) \nN^{\eps}(v_{m,\eps})\\
	&+ \theta_{m} \big(\|v_{m,\eps}\|\big) \nN^{\eps}(v_{m,\eps}) - \theta_{m} \big(\|v_{m}\|\big) \nN(v_{m})\\
	&+ \theta_{m} \big(\|v_{m}\|\big) \nN(v_{m}) - \theta_{m} \big(\|u_{m}\|\big) \nN(u). 
	\end{split}
	\end{equation*}
	By Proposition~\ref{pr:uniform_stable}, we have
	\begin{equation*}
	\sup_{\eps \in [0,1]} \|u_{m,\eps} - v_{m,\eps}\|_{L_{\omega}^{\rho_0} \xX(0,T_0)} \leq \|u_0 - v_0\|_{L_{\omega}^{\rho_0}L_{x}^{2}}. 
	\end{equation*}
	The bound for the first and third term above then follows from this uniform stability and the Strichartz estimate \eqref{eq:Strichartz_2}. As for the second term, we can get the desired bound from Proposition~\ref{pr:lwpinhigh} and the uniform boundedness in Proposition~\ref{pr:uniform_e_m} with $\rho > \rho_0$ in exactly the same way as above. Note that it is only simpler in that we do not need to decompose $[0,T_0]$ into smaller subintervals, as we already have a priori control of the difference of the arguments on the whole interval $[0,T_0]$. We can then conclude that the limit $u_{m}$ solves the truncated critical equation \eqref{eq:trun_critical}. 
\end{proof}

%
%
%

\section{Removing the truncation}
\label{sec:converge_m}

\subsection{Proof of Theorem~\ref{th:main}}

In this section, we prove Theorem~\ref{th:main}. We will show that the sequence of solutions $\{u_{m}\}$ is Cauchy in $L_{\omega}^{\rho_{0}}\xX(0,T_0)$, and that the limit $u$ satisfies the corresponding Duhamel's formula for equation \ref{eq:main_eq}. The removal of the truncation $m$ relies crucially on the uniform bound in Proposition \ref{pr:uniform_e_m}. 

For each $m$ and $\eps$, we let
\begin{equation*}
\begin{split}
\tau_{m} &= T_0 \wedge \inf \big\{ t \leq T_0: \|u_{m}\|_{\xX_{2}(0,t)} \geq m-1 \big\},\\
\tau_{m,\eps} &= T_0 \wedge \inf \big\{ t \leq T_0: \|u_{m,\eps}\|_{\xX_{2}(0,t)} \geq m \big\}. 
\end{split}
\end{equation*}
Note that $\tau_m$ is different from $\tau_{m,\eps}$ by simply setting $\eps=0$. We make $\tau_m$ the stopping time when hitting $m-1$ instead of $m$ to simplify the arguments in Lemma~\ref{le:st} below. 

It has been shown in \cite[Lemma 4.1]{BD} that for every $\eps$ and $m$, $\tau_{m,\eps} \leq \tau_{m+1,\eps}$ almost surely, and 
\begin{equation} \label{eq:st_dd}
u_{m,\eps} = u_{m+1,\eps} \quad \text{in}\; \xX(0,\tau_{m,\eps})
\end{equation}
almost surely. Note that the null set excluded may be different when $\eps$ changes, and it does not exclude the possibility that the set of $\omega \in \Omega$ for which \eqref{eq:st_dd} is true for all $\eps$ has probability $0$! Nevertheless, we have a similar statement for the limit $u_{m}$. 

\begin{lem} \label{le:st}
	For every $m$, we have $u_{m} = u_{m+1}$ in $\xX(0,\tau_{m} \wedge \tau_{m+1})$ almost surely. 
\end{lem}
\begin{proof}
	Fix $m$ arbitrary. Since $u_{k,\eps} \rightarrow u_{k}$ in $L_{\omega}^{\rho_0}\xX(0,T_0)$ for $k=m,m+1$, there exists $\Omega' \subset \Omega$ with full measure and a sequence $\eps_n \rightarrow 0$ such that $u_{k,\eps_n} \rightarrow u_k$ on $\Omega'$. We then write
	\begin{equation*}
	|u_{m} - u_{m+1}| \leq |u_{m} - u_{m,\eps_n}| + |u_{m,\eps_n}-u_{m+1,\eps_n}| + |u_{m+1,\eps_n}-u_{m}|. 
	\end{equation*}
	On $\Omega'$, the first and third term can be made arbitrarily small when $n$ is large. Also, the convergence of $u_{k,\eps_n}$ to $u_k$ and the definition of the stopping times implies that $\tau_{m,\eps_n} \geq \tau_m$ for all large $n$. By \eqref{eq:st_dd}, we then have $u_{m,\eps_n} = u_{m+1,\eps_n}$ in $\xX(0,\tau_{m} \wedge \tau_{m+1})$ almost surely. This shows that on a set of full measure, $|u_{m} - u_{m+1}|$ can be made arbitrarily small, thus concluding the proof. 
\end{proof}

\begin{proof} [Proof of Theorem~\ref{th:main}]
	We are now ready to prove Theorem~\ref{th:main}. We first show that $\{u_{m}\}$ is Cauchy in $L_{\omega}^{\rho_0} \xX(0,T_0)$. To see this, we fix $\eta>0$ arbitrary. By Proposition~\ref{pr:uniform_e_m} and that $u_{m,\eps} \rightarrow u_{m}$ in $L_{\omega}^{\rho_0} \xX(0,T_0)$, we have
	\begin{equation} \label{eq:trun_small}
	\Pr \Big( \|u_{m}\|_{\xX(0,T_0)} \geq K \Big) \leq \frac{C}{K^{\rho_0}}
	\end{equation}
	for every $m$ and every $K$. Now, for every $m$, $m'$ and $K$, we let
	\begin{equation*}
	\Omega_{m,m'}^{K} = \Big\{ \|u_{m}\|_{\xX(0,T_0)} \geq K \Big\} \cup \Big\{ \|u^{m'}\|_{\xX(0,T_0)} \geq K \Big\}. 
	\end{equation*}
	By \eqref{eq:trun_small} and the uniform boundedness of $\|u_{m}\|_{L_{\omega}^{\rho_0}\xX(0,T_0)}$ for all $\rho$, we know there exists a large $K$ such that
	\begin{equation} \label{eq:truncation_small}
	\big\|\1_{\Omega_{m,m'}^{K}} \big(u_{m} - u_{m'} \big) \big\|_{L_{\omega}^{\rho_0} \xX(0,T_0)} \leq \eta
	\end{equation}
	for every $m$ and $m'$. We now fix this $K$ such that \eqref{eq:truncation_small}, and take any $m, m' > K+1$. By Lemma~\ref{le:st}, we have
	\begin{equation*}
	u_{m} = u_{m'} \qquad \text{on} \phantom{1} (\Omega_{m,m'}^{K})^{c}. 
	\end{equation*}
	This shows that $\|u_{m} - u_{m'}\|_{L_{\omega}^{\rho_0}\xX(0,T_0)} < \eta$ whenever $m,m' > K+1$. Hence $\{u_{m}\}$ is Cauchy and converges to a limit $u$ in $L_{\omega}^{\rho_0}\xX(0,T_0)$. 
	
	It remains to show that $u$ satisfies the Duhamel's formula \eqref{eq:duhamel_main}. This part follows similarly as the proof for the equation of $u_{m}$, but only easier. We need to show  that each term on the right hand side of \eqref{eq:duhamel_trun_sub} converges to the corresponding term in $L_{\omega}^{\rho_0}\xX(0,T)$ as $m \rightarrow +\infty$. The convergence of
	\begin{equation*}
	\int_{0}^{t} \sS(t-s) u_{m}(s) {\rm d} W_s \quad \text{and} \quad \int_{0}^{t} \sS(t-s) \big( F_{\Phi} u_{m}(s) \big) {\rm d}s
	\end{equation*}
	follows immediately from the bounds in Propositions~\ref{pr:pre_stochastic} and~\ref{pr:pre_correction} and that $u_{m} \rightarrow u$ in $L_{\omega}^{\rho_0}(0,T_0)$. We now turn to the term with nonlinearity. We first note that Proposition~\ref{pr:uniform_e_m} holds for all $\rho \geq \rho_0$. Hence, for every $\rho \geq \rho_0$, we have $u \in L_{\omega}^{\rho}\xX(0,T_0)$ and
	\begin{equation} \label{eq:high_rho}
	\|u_{m}-u\|_{L_{\omega}^{\rho}(0,T_0)} \rightarrow 0
	\end{equation}
	as $m \rightarrow +\infty$. As for the nonlinearity itself, we have the pointwise bound
	\begin{equation*}
	\Big|\theta_{m} \big(\|u_{m}\|\big) \nN(u_{m}) - \nN(u)\Big| \leq C |u_{m} - u| \big( |u_{m}|^{4} + |u|^{4} \big) + \big|\theta_{m}\big(\|u_{m}\|\big)-1\big| \cdot |u|^{5}.  
	\end{equation*}
	Applying the Strichartz estimate \eqref{eq:Strichartz_2} and then taking the $\rho_0$-th moment on both sides, we get
	\begin{equation*}
	\begin{split}
	&\phantom{111}\EE \Big\| \int_{0}^{t} \sS(t-s) \Big(\theta_{m} \big(\|u_{m}\|_{\xX_{2}(0,s)}\big) \nN\big(u_{m}(s)\big) - \nN \big(u(s)\big)\Big) {\rm d}s \Big\|_{\xX(0,T_0)}^{\rho_0}\\
	&\leq C \EE \Big( \|u_{m} - u\|^{\rho_0} \big( \|u_{m}\|^{4\rho_0} + \|u\|^{4\rho_0} \big) + \|u\|^{5 \rho_0} \sup_{t \in [0,T_0]} |\theta_{m}\big( \|u_{m}\|_{\xX_{2}(0,t)} \big)-1|^{\rho_0} \Big). 
	\end{split}
	\end{equation*}
	Here, the norms of $\|u_{m}\|$ and $\|u\|$ are $\xX(0,T_0)$ unless otherwise specified. For the first term above, applying H\"{o}lder and using the fact that \eqref{eq:high_rho} holds for every $\rho < \rho_0$, we get
	\begin{equation*}
	\EE \Big( \|u_{m} - u\|^{\rho_0} \big( \|u_{m}\|^{4\rho_0} + \|u\|^{4\rho_0} \big) \Big) \rightarrow 0
	\end{equation*}
	as $m \rightarrow +\infty$. As for the second term, for the same reason, it suffices to show that
	\begin{equation} \label{eq:vanish}
	\EE \sup_{t \in [0,T_0]} |\theta_{m}\big( \|u_{m}\|_{\xX_{2}(0,t)} \big)-1|^{2\rho_0} \rightarrow 0. 
	\end{equation}
	Note that
	\begin{equation*}
	|\theta_{m}\big( \|u_{m}\|_{\xX_{2}(0,t)} \big)-1| \leq 2 \times \Pr \big( \|u_{m}\|_{\xX_{2}(0,T_0)} \geq m \big)
	\end{equation*}
	for all $t \in [0,T_0]$, so \eqref{eq:vanish} follows from the uniform (in $m$) boundedness of $\|u_{m}\|_{L_{\omega}^{\rho_0}\xX(0,T_0)}$. The proof of the main theorem is then complete. 
\end{proof}

\subsection{Stability of the solution}
\label{sec:stability}

The solution $u$ we constructed above is stable under perturbation of initial data. This is the content of the following proposition. 

\begin{prop} \label{pr:truestable}
	For every $\eta>0$, there exists $\kappa>0$ such that for every $u_0, v_0 \in L_{\omega}^{\infty} L_{x}^{2}$ with
	\begin{equation*}
	\|u_0\|_{L_{\omega}^{\infty}L_{x}^{2}} \leq \delta_0, \quad \|v_0\|_{L_{\omega}^{\infty}L_{x}^{2}} \leq \delta_0, \quad \|u_0 - v_0\|_{L_{\omega}^{\rho_0}L_{x}^{2}} \leq \kappa, 
	\end{equation*}
	the corresponding solutions $u$ and $v$ given as in Theorem~\ref{th:main} satisfies
	\begin{equation*}
	\|u-v\|_{L_{\omega}^{\rho_0}\xX(0,T_0)} < \eta. 
	\end{equation*}
\end{prop}
\begin{proof}
	We have
	\begin{equation*}
	\|u-v\|_{L_{\omega}^{\rho_0}\xX(0,T_0)} \leq \|u-u_{m}\|_{L_{\omega}^{\rho_0}\xX(0,T_0)} + \|u_{m}-v_{m}\|_{L_{\omega}^{\rho_0}\xX(0,T_0)} + \|v_{m}-v\|_{L_{\omega}^{\rho_0}\xX(0,T_0)}. 
	\end{equation*}
	Since $u_{m} \rightarrow u$ in $L_{\omega}^{\rho_0}\xX(0,T_0)$, there exists $m$ sufficiently large such that both the first and third term above are smaller than $\frac{\eta}{3}$. In addition, this $m$ depends only on the size $\delta_0$ of the initial data but not the concrete profile. As for the second term, by the convergence of $u_{m,\eps} \rightarrow u_{m}$ and $v_{m,\eps} \rightarrow v_{m}$ and the uniform in $\eps$ stability in Proposition~\ref{pr:uniform_stable}, we can choose $\kappa$ sufficiently small so that it is also smaller than $\frac{\eta}{3}$. The proof is then complete. 
\end{proof}

\appendix

\endappendix

\bibliographystyle{Martin}
\bibliography{Refs}

\end{document}